\documentclass{amsart}
\usepackage[utf8]{inputenc}
\usepackage{amsmath,amssymb,enumitem,mathrsfs,amsthm,graphicx,multicol}
\usepackage{comment}
\usepackage{url}
\usepackage{subfigure}
\usepackage{hyperref}
\usepackage{cite}
\usepackage[dvipsnames]{xcolor}
\definecolor{light-gray}{gray}{0.95}

\numberwithin{equation}{section}

\newcommand{\R}{\mathbb{R}}

\newcommand{\E}{\mathbb{E}}

\newcommand{\ut}{u^\dagger}

\newcommand{\calK}{\mathcal{K}}

\DeclareMathOperator{\Var}{Var}

\DeclareMathOperator{\bigo}{O}

\DeclareMathOperator{\Tr}{Tr}

\DeclareMathOperator{\bias}{bias}
\DeclareMathOperator{\var}{var}
\DeclareMathOperator{\Ran}{Ran}
\DeclareMathOperator{\Dom}{Dom}
\DeclareMathOperator{\iid}{i.i.d.}

\newcommand{\abs}[1]{\left |#1\right |}
\newcommand{\norm}[1]{\left \|#1\right\|}

\newcommand{\paren}[1]{\left(#1\right)}
\newcommand{\bracket}[1]{\left[#1\right]}
\newcommand{\set}[1]{\left\{#1\right\}}
\newcommand{\inner}[2]{\left\langle #1,#2\right\rangle}
\newcommand{\eps}{\epsilon}

\newtheorem{thm}{Theorem}[section]
\newtheorem{lem}[thm]{Lemma}
\newtheorem{cor}[thm]{Corollary}
\newtheorem{prop}[thm]{Proposition}
\newtheorem{rem}[thm]{Remark}

\newtheorem{assume}{Assumption}[]
\newtheorem{assumealt}{Assumption}[assume]

\newenvironment{assumep}[1]{
  \renewcommand\theassumealt{#1}
  \assumealt
}{\endassumealt}

\theoremstyle{definition}

\title[Averaging and Filtering]{Iterate Averaging, the Kalman Filter, and 3DVAR for Linear Inverse Problems}
\author[Jones]{Felix G. Jones}
\author[Simpson]{Gideon Simpson}
\email{grs53@drexel.edu}
\address{Department of Mathematics, Drexel University, Philadelphia, PA, 19103, USA}
\date{\today}

\begin{document}

\maketitle

\begin{abstract}
    It has been proposed that classical filtering methods, like the Kalman filter and 3DVAR, can be used to solve linear statistical inverse problems.  In the work of Iglesias, Lin, Lu, \& Stuart (2017), \cite{Iglesias_Lin_Lu_Stuart_2017}, error estimates were obtained for this approach.  By optimally tuning a regularization parameter in the filters, the authors were able to show that the mean squared error could be systematically reduced.

    Building on the aforementioned work of Iglesias, Lin, Lu, \& Stuart, we prove that by (i) considering the problem in a weaker norm and (ii) applying simple iterate averaging of the filter output, 3DVAR will converge in mean square, unconditionally on the choice of parameter.  Without iterate averaging, 3DVAR cannot converge by running additional iterations with a  fixed choice of parameter.  We also establish that the Kalman filter's performance in this setting cannot be improved through iterate averaging.  We illustrate our results with numerical experiments that suggest our convergence rates are sharp.
\end{abstract}


\section{Introduction}

The focus of this work is on the inverse problem
\begin{equation}
\label{e:invproblem}
y = Au^\dagger + \eta,
\end{equation}
where, given the noisy observation $y$ of $A u^\dagger$, we wish to infer $u^\dagger$.  In our setting, $A:X\to Y$ is a compact operator between separable Hilbert spaces and $\eta \sim N(0, \gamma^2 I)$ is white noise, modelling measurement error.  This problem is well-known to be ill-posed in the infinite dimensional setting, as $A$ has an unbounded inverse.  Methods of solution include the use of regularized Moore-Penrose inverses and, subject to the introduction of a prior, Bayesian formulations,  \cite{Pereverzev, Heinz, ghanem_handbook_2017, sullivan2015introduction, stuart_inverse_2010, knapik2011bayesian,Ghosal_vanderVaart_2017}.

In  \cite{Iglesias_Lin_Lu_Stuart_2017}, a key inspiration for the present work, Iglesias, Lin, Lu, \& Stuart considered two classical
filtering algorithms, the Kalman filter and 3DVAR, with the goal of using them to solve
\eqref{e:invproblem}. The filtering methodology for \eqref{e:invproblem} requires the introduction, conceptually, of the artificial dynamical system
\begin{subequations}
\label{e:artificial1}
\begin{align}
    u_n &= u_{n-1},\quad u_0 = u^\dagger,\\
    y_n & = Au_n + \eta_n, \quad \eta_n \overset{\iid}{\sim} N(0, \gamma^2 I).
\end{align}
\end{subequations}
Here, at algorithmic time step $n$, $u_n$ is the quantity of interest, and $y_n$ is the noisy observation.  Having ascribed a notion of time to the problem, we can then apply a filter.  This provides a mechanism for estimating  $u^\dagger$  in \eqref{e:invproblem} in an online setting, where a sequence of i.i.d. observations, $\{y_n\}$, is available.  This corresponds to ``Data Model 1'' of \cite{Iglesias_Lin_Lu_Stuart_2017}.

Amongst the key results of \cite{Iglesias_Lin_Lu_Stuart_2017}, reviewed in detail below, is that under sufficiently strong assumptions, the Kalman filter will recover the truth in mean square, unconditionally on the choice of the scalar regularization parameter.  Under somewhat weaker assumptions, the error will only be bounded, though through minimax selection of a scalar parameter, an optimal error can be achieved for a given number of iterations, allowing the error to be driven to zero.

3DVAR is a simplification of Kalman that is demonstrated to have, at best, bounded error, though, again, through minimax parameter tuning, it can perform comparably to Kalman.   Kalman is more expensive than 3DVAR, as it requires updating an entire covariance operator at each iteration. For finite dimensional approximations, this may require costly matrix-matrix multiplications at each iterate.  

Here, by working in a weaker norm and averaging the iterates, we are able to establish that 3DVAR will unconditionally converge in mean square for all admissible filter parameters.  Such weaker convergence was also considered in \cite{ding2018weak}, for a related problem on 4DVAR.  Further, we show that this simple  iterate averaging {\it cannot} improve the performance of the Kalman filter.

\subsection{Filtering Algorithms}

The Kalman filter is a probabilistic filter that estimates a Gaussian distribution, $N(m_n, C_n)$, for $\ut$ at each iterate.  Given a starting  mean and covariance, $m_0$ and $C_0$, the updates are as follows:
\begin{subequations}
~\label{e:Kalman}
\begin{align}
    m_{n} & = K_n y_n + (I-K_n A)m_{n-1},\\
    C_n &= (I - K_n A) C_{n-1},\\
    K_n & = C_{n-1} A^\ast (A C_{n-1}A^\ast + \gamma^2 I)^{-1}.
 \end{align}
\end{subequations}
Here, $K_n$ is the so-called ``Kalman gain.''  $m_n$ is a point estimate of $\ut$.

While Kalman is a probabilistic filter, 3DVAR is not. It is obtained by applying Kalman with a static covariance operator $C_n = \frac{\gamma^2}{\alpha} \Sigma$ for some predetermined operator $\Sigma$: 
\begin{subequations}
\label{e:3DVAR}
\begin{align}
    u_{n} &= \calK y_n + (I - \calK A)u_{n-1},\\
    \calK &= (A^\ast A + \alpha \Sigma^{-1})^{-1}A^\ast.
\end{align}
\end{subequations}
  We refer the reader to \cite{law_data_2015,sullivan2015introduction,shumway_time_2011, humpherys_fresh_2012}, and references therein, for a thorough discussion and analysis of these classical filtering methods and their extensions.

 Indeed, several important extensions of these classical methods that have appeared in the literature have also been directly applied to statistical inverse problems like \eqref{e:invproblem}, along with its nonlinear variation, $y = \mathcal{G}(u^\dagger) + \eta$.   In particular, the ensemble Kalman filter (EnKF), using an ensemble of replicas of the problem,  has been successfully applied to solve such problems in \cite{iglesias2013ensemble,iglesias2016}.  See, for instance, \cite{law_data_2015, schillings2017a,schillings2017b}, for additional details and analysis of EnKF.  We also mention \cite{ding2018weak}, which uses similar ideas with 4DVAR.
 
 Continuous in time analogs of these methods and problems also exist, resulting in the Kalman-Bucy filter and continuous in time 3DVAR, \cite{law_data_2015, oksendal2003stochastic,sullivan2015introduction}.  In \cite{lu2021asymptotical}, these were used to solve the continuous in time analog of \eqref{e:artificial1}
\begin{subequations}
    \label{e:continv}
 \begin{align}
 du &= 0,\\
 dy & = Au dt + d\eta,
 \end{align}
\end{subequations}
where $\eta(t)$ is now a Weiner process in the appropriate function space, \cite{da2014stochastic}.

\subsection{Key Assumptions and Prior Results}

In \cite{Iglesias_Lin_Lu_Stuart_2017}, the following assumptions were invoked.
\begin{assume}
\label{a1}
~
\begin{enumerate}
    \item $C_0=\frac{\gamma^2}{\alpha}\Sigma$ with $\Ran(\Sigma^{\frac{1}{2}})\subset \Dom(A)$, $\alpha>0$, and $\Sigma$ a self-adjoint positive definite trace class operator with $\Sigma^{-1}$ densely defined. 
    
    \item $\Sigma$ induces a Hilbert scale, and there exist constants  $C> 1$, $\nu>0$ such that $A$ induces an equivalent norm:
    \begin{equation}
    \label{e:Hscale}
        C^{-1}\|x\|_\nu\leq \|Ax \|\leq C \|x\|_\nu, \quad \|\bullet\|_\nu = \|\Sigma^{\frac{\nu}{2}} \bullet\|.
    \end{equation}

    \item The initial error is sufficiently ``smooth,''
    \begin{equation}
    \label{e:datacond}
     m_0 -\ut \in  \Dom(\Sigma^{-\frac{s}{2}}), \quad 0\leq s \leq \nu + 2,
    \end{equation}
    where we replace $m_0$ with $u_0$ in the case of 3DVAR in the above expression.
    
\end{enumerate}
\end{assume}
Under this first set of assumptions, Iglesias et al. established
\begin{thm}[Theorem 4.1 of \cite{Iglesias_Lin_Lu_Stuart_2017}]
\label{t:Iglesias41}
The Kalman filter admits the  mean square error bound
\begin{equation*}
    \E[\|m_n - \ut\|^2]\lesssim \paren{\frac{n}{\alpha}}^{-\frac{s}{\nu + 1}} + \frac{\gamma^2}{\alpha}\Tr\Sigma
\end{equation*}
\end{thm}
and
\begin{thm}[Theorem 5.1 of \cite{Iglesias_Lin_Lu_Stuart_2017}]
\label{t:Iglesias51}
3DVAR admits the mean square error bound
\begin{equation*}
    \E[\|u_n - \ut\|^2]\lesssim \paren{\frac{n}{\alpha}}^{-\frac{s}{\nu + 1}} + \frac{\gamma^2}{\alpha}\Tr\Sigma \log n.
\end{equation*}
\end{thm}
At fixed values of $\alpha$, Theorems \ref{t:Iglesias41} and \ref{t:Iglesias51} preclude convergence, and, in the case of 3DVAR, the error may even grow.  However, there are two free parameters: the number of iterations $n$ and the regularization parameter $\alpha$.  Indeed,  within a Bayesian framework, $\alpha$ can be interpreted as the strength of a prior relative to a likelihood.  For a fixed number of iterations, $n$, $\alpha$ can be tuned to minimize the error.  Indeed, the error can be made arbitrarily small by selecting a sufficiently large $n$ with the optimal $\alpha$.

However, in both Theorems \ref{t:Iglesias41} and \ref{t:Iglesias51}, there is an unknown constant.  If the error at the given, optimal choice of $\alpha$ for a given $n$ is inadequate, one must obtain additional data, update the value of $\alpha$, and rerun the algorithm.  A benefit of the present work is that, by using iterate averaging, the error of 3DVAR can always be reduced with additional iterates, without necessarily altering $\alpha$ and discarding previously computed iterations.  We will revisit the minimax estimates under a simultaneous diagonalization assumption.

Indeed, stronger results were obtained in \cite{Iglesias_Lin_Lu_Stuart_2017} subject to the simultaneous diagonalization assumption:
\begin{assume}
\label{a2}
~
\begin{enumerate}
    \item
$\Sigma$ and $A^\ast A$ simultaneously diagonalize against the set  $\{\varphi_i\}$ with respective eigenvalues $\sigma_i$ and $a_i^2$, and these eigenvalues satisfy
\begin{equation}
\label{e:evals1}
    \sigma_i = i^{-1 - 2\eps}, \quad a_i \asymp i^{-p}, \quad \eps>0, \quad p>0.
\end{equation}

\item $m_0 = 0$ (or $u_0$ in 3DVAR) and $\ut$ satisfies, for $0<\beta \leq 1 + 2\eps + 2p$,
\begin{equation}
\label{e:sobolev1}
    \sum_{i=1}^\infty i^{2\beta}|u^{\dagger}_i|^2 <\infty.
\end{equation}
\end{enumerate}
\end{assume}

With this, Iglesias et al. obtain
\begin{thm}[Theorem 4.2 of \cite{Iglesias_Lin_Lu_Stuart_2017}]
\label{t:Iglesias2}
Under Assumption \ref{a2}, for the Kalman filter,
\begin{equation*}
\E[\|m_n - \ut\|^2]\lesssim  \paren{\frac{n}{\alpha}}^{-\frac{2\beta}{1+2\eps + 2p}} + \gamma^2 n^{-\frac{2\eps}{1 + 2\eps + 2p}}\alpha^{-\frac{1 + 2p}{1 + 2\eps + 2p}}
\end{equation*}
\end{thm}
and
\begin{thm}[Theorem 5.2 of \cite{Iglesias_Lin_Lu_Stuart_2017}]
\label{t:Iglesias6}
Under Assumption \ref{a2}, for 3DVAR,
\begin{equation*}
    \E[\|u_n - \ut\|^2]\lesssim \paren{\frac{n}{\alpha}}^{-\frac{2\beta}{1+2\eps + 2p}} + C\gamma^2\alpha^{-\frac{1 + 2p}{1 + 2\eps + 2p}}.
\end{equation*}
\end{thm}
Now the Kalman filter will converge at any choice of parameter, while 3DVAR has at worst a bounded error.  Again, $\alpha$ can be tuned so as to obtain a minimax convergence rate.    Indeed, in the setting where one has a fixed number of $n$ samples, at the optimal value of $\alpha$, Theorems \ref{t:Iglesias2} and \ref{t:Iglesias6} lead to the estimates (also found in \cite{Iglesias_Lin_Lu_Stuart_2017}):
\begin{align}
\label{e:kalmanminimax}
     \E[\|m_n - \ut\|^2]\lesssim & n^{-\frac{2\beta}{1+2\beta + 2p}},\\
\label{e:3dvarminimax}     
     \E[\|u_n - \ut\|^2]\lesssim & n^{-\frac{2\beta}{1 + 2\beta + 2p + 2\eps}}\log n,
\end{align}
where the first expression is for Kalman and the second is for 3DVAR.  Similar expressions are also available in the general case for Theorems \ref{t:Iglesias41} and \ref{t:Iglesias51}.

Thus far, we have discussed the study of problem \eqref{e:invproblem} in a sequential setting, where the data, $\{y_n\}$, is assimilated one sample at a time.  In some settings, a static, fixed, number of samples, $n$, may be available together.  Instead of \eqref{e:invproblem}, we might then examine 
\begin{equation}
\label{e:batch}
\begin{split}
    \bar{y}_n = A u^\dagger + \bar{\eta}_n, \quad \bar{\eta}_n \sim N(0, \tfrac{\gamma^2}{n} I),\\
    \bar{y}_n = \frac{1}{n}\sum_{k=1}^n y_k, \quad \bar{\eta}_n = \frac{1}{n}\sum_{k=1}^n \eta_k. 
\end{split}
\end{equation}
The variance of the noise has been reduced by a factor of $n$.  This can be solved using a regularized approximation of $A^+$ to obtain $\bar{u}_{n, \alpha}$.  Under suitable assumptions and identifying the optimal $\alpha = \alpha_\star(n)$, one can obtain (see, for instance,  \cite{cavalier2008nonparametric, Pereverzev,mair1996statistical,mathe_optimal_2001,knapik2011bayesian,van_rooij_asymptotic_1996})
\begin{equation}
\label{e:minimax1}
    \E[\|\bar{u}_{n,\alpha_\star(n)} - u^\dagger\|^2] \lesssim n^{-\frac{ 2\beta}{1 + 2p + 2\beta}}
\end{equation}
This precisely corresponds to the minimax solution of Kalman, \eqref{e:kalmanminimax}, while there is a loss for 3DVAR, \eqref{e:3dvarminimax}.  Note that this is only for the $t=0$ norm.  A generalization to the $t <0$ norm is covered in \cite{mair1996statistical} and for $t(1+2\eps)\leq 2 p$ in \cite{mathe_optimal_2001}.  As we are principally interested in the general $t\geq 0$ case, we state and prove our own version of theorem below using a spectral cutoff regularization.

\subsection{Main Results}
The main results of this paper are contained in the following theorems.  

First, we have the elementary result that 3DVAR, without averaging, cannot converge at fixed parameter choices:
\begin{thm}
	\label{t:3dvarnoconverge}
	Under Assumption \ref{a1} in dimension one, if $u_n$  is generated by 3DVAR, then
	\begin{equation*}
	\E[|u_n - \ut|^2]\geq \gamma^2 \mathcal{K}^2.
	\end{equation*}
\end{thm}
As the method cannot converge in dimension one, it has no hope of converging in higher dimensions.    By time averaging,
\begin{equation}
\label{e:iteravg1}
\bar{u}_n = \frac{1}{n}\sum_{k=1}^nu_k
=\frac{1}{n}u_n + \frac{n-1}{n}\bar{u}_{n-1},
\end{equation}
we can obtain convergence for all $\alpha>0$:
\begin{thm}
\label{t:msegeneral1}
Under Assumption \ref{a1}, fix $t\in [0,\nu]$ and $\tau_{\rm v} \in [0, 1]$, and, having set these indices, assume that $\Sigma^{t+1 - \tau_{\rm v}(1+\nu)}$ is trace class. Then
\begin{equation*}
\E[\|\bar{u}_n-u^\dagger\|_t^2] \lesssim 
 \paren{\frac{n}{\alpha}}^{-\frac{s+t}{1+\nu}}\|z_0\|^2 + \frac{\gamma^2}{\alpha}\Tr(\Sigma^{t+1 - \tau_{\rm v}(1+\nu)}) \paren{\frac{n}{\alpha}}^{-\tau_{\rm v}}
\end{equation*}
where $z_0$ is the solution to 
\begin{equation}
\label{e:auxbias}
    \Sigma^{-\frac{1}{2}}(u_0 - \ut) = (B^\ast B)^{
    \frac{s-1}{2(1+\nu)}}z_0
\end{equation}
and $B = A\Sigma^{\frac{1}{2}}$.
\end{thm}
We will repeatedly make use of the operator
\begin{equation}
\label{e:Bdef}
    B = A\Sigma^{\frac{1}{2}}
\end{equation}
throughout this work.  The existence of $z_0$ in \eqref{e:auxbias} is a consequence of Assumption \ref{a1} on the initial error and an equivalence of spaces result encapsulated in Proposition \ref{p:hilbertscale}, given below.

The motivation for time averaging comes from two related problems.  First, formally, \eqref{e:3DVAR}, has the structure of an AR(1) process, \cite{shumway_time_2011}.  Under typical assumptions, an AR(1) process will not  converge to a fixed value, but instead, sample an invariant distribution. Consequently, the time average will converge to the mean, with respect to this invariant distribution.  Another motivation comes from the stochastic root finding problem and the Robbins-Monro algorithm.  In \cite{polyak1992acceleration}, Polyak \& Juditsky proved that by time averaging the sequence of of estimates generated by Robbins-Monro, the convergence rate could be improved.  See, also,  \cite{kushner_stochastic_2003}.

As a consequence of Theorem \ref{t:msediagonal1}, we will have unconditional mean squared convergence of the iterate averaged value, $\bar{u}_n$, provided:
\begin{itemize}
    \item We study the problem in a sufficiently weak weighted space ($t>0$) and/or have sufficiently smooth data ($s>0$);
    
    \item $\Sigma$ has a sufficiently well behaved spectrum, allowing $\tau_{\rm v}>0$.  Note that taking $\tau_{\rm v} = t/(1+\nu)$ will not require additional assumptions on $\Sigma$, but will require $t>0$ for convergence.
\end{itemize}
We emphasize that iterate averaging is a {\it post-processing} step, requiring no modification of the underlying 3DVAR iteration.

We introduce a modified version of Assumption \ref{a2},
\begin{assumep}{$2'$}
\label{a2p}
~
\begin{enumerate}
    \item
$\Sigma$ and $A^\ast A$ simultaneously diagonalize against the set  $\{\varphi_i\}$ with respective eigenvalues $\sigma_i$ and $a_i^2$, and these eigenvalues satisfy
\begin{equation}
\label{e:evals2}
    \sigma_i \asymp i^{-1 - 2\eps}, \quad a_i \asymp i^{-p}, \quad \eps>0, \quad p>0.
\end{equation}

\item For $\beta\geq 0$, the initial error, $u_0 - u^\dagger$, satisfies the condition
\begin{equation}
\label{e:sobolev2}
    \sum_{i=1}^\infty i^{2\beta}|u_{0,i} - u^{\dagger}_i|^2 <\infty.
\end{equation}
\end{enumerate}
\end{assumep}
Condition \eqref{e:sobolev2} on the initial error will automatically be satisfied if $u_0$ and $u^\dagger$ are, separately, sufficiently smooth.  The assumptions of \eqref{e:evals2} and \eqref{e:sobolev2} are equivalent to those of \eqref{e:Hscale} and \eqref{e:datacond} under the identifications:
\begin{equation*}
\label{e:diagrelation}
    \nu(1+2\eps) = {2p}, \quad s(1+2\eps) = {2\beta}.
\end{equation*}
In contrast to Assumption \ref{a2}, no upper bound on $\beta$ is necessary.

\begin{thm}
\label{t:msediagonal1}
Under Assumption \ref{a2p}, and having fixed a choice of  $\norm{\bullet}_t$ norm with $t\geq 0$, assume $\tau_{\rm b}, \tau_{\rm v}\in [0,1]$ satisfy
\begin{subequations}
\begin{align}
    \label{e:taucondb}
    \tau_{\rm b} &\leq  \frac{t(1+2\eps) + 2\beta }{2(1+2\eps + 2p)} \equiv \bar{\tau}_{\rm b}\\
    \label{e:taucondv}
    \tau_{\rm v} &< \frac{t(1+2\eps) + 2\eps }{1+2\eps + 2p} \equiv \bar{\tau}_{\rm v}
\end{align}
\end{subequations}
then, 
\begin{equation*}
\begin{split}
    \E[\| \bar{u}_n - \ut\|_t^2]\lesssim  \paren{\frac{n}{\alpha}}^{-2\tau_{\rm b}} +\frac{\gamma^2}{\alpha}\paren{\frac{n}{\alpha}}^{-\tau_{\rm v}}
\end{split}
\end{equation*}

\end{thm}

While our results in both the general and diagonal case establish unconditional convergence for any choice of $\alpha$ for the iterate averaged 3DVAR, in a practical setting, there may only be $n$ iterates available.  One might then ask how well iterate averaged 3DVAR behaves if, at fixed $n$, we choose the optimal $\alpha$, and how this would compare to the minimax solution of \eqref{e:batch}.  Focusing on the diagonal case, for comparison, we have the following result for the the minimax solution of \eqref{e:batch}:
\begin{thm}
    \label{t:batchminimax1}
    Under Assumption \ref{a2p} with $u_0 = 0$ in \eqref{e:sobolev2}, if \eqref{e:batch} is solved using a spectral cutoff with regularization $\alpha$ in the $t\geq 0$ norm, then at the optimal value of $\alpha= \alpha_{\star}(n)$,
    \begin{equation*}
        \E[\|\bar{u}_{n, \alpha_\star(n)} - u^\dagger\|_t^2]\lesssim \begin{cases}
            n^{-\frac{t(1+2\eps) + 2\beta}{1 + 2p + 2\beta}} & 1+2p \neq t(1+2\eps)\\
            n^{-1}\log n &1+2p = t(1+2\eps)
        \end{cases}
    \end{equation*}
\end{thm}
This is consistent with \eqref{e:minimax1} and the results in \cite{mair1996statistical,mathe_optimal_2001}.  Then, looking at the minimax solution of 3DVAR, we obtain for two particular regimes:
\begin{cor}
    \label{c:minimaxdiag1}
    With the same assumptions as Theorem \ref{t:msediagonal1}, first, assume $\bar\tau_{\rm b}, \bar\tau_{\rm v}\leq 1$.  Taking $\tau_{\rm b} = \bar\tau_{\rm b}$ and $\tau_{\rm v} = (1-\theta) \bar\tau_{\rm v}$ for $\theta \in (0,1]$,
    \begin{equation*}
        \E[\| \bar{u}_n - \ut\|_t^2] \lesssim 
        \theta^{-\frac{2\beta + t(1+2\eps)}{1+2p + 2\beta + \theta[t(1+2\eps) + 2\eps]}}n^{-\frac{t(1+2\eps)+2\beta}{1+2p + 2\beta + \theta[t(1+2\eps) + 2\eps]}}.
    \end{equation*}

    If, instead, $\bar\tau_{\rm b}, \bar\tau_{\rm v}>1$, then, taking $\tau_{\rm b}= \tau_{\rm v}= 1$,
    \begin{equation*}
        \E[\| \bar{u}_n - \ut\|_t^2] \lesssim  n^{-1}
    \end{equation*}
\end{cor}

Consequently:
\begin{itemize}
\item At $t=0$, in the first case,
\begin{equation*}
    \E[\| \bar{u}_n - \ut\|_t^2] \lesssim \theta^{-\frac{2\beta }{1+2p + 2\beta + 2\eps \theta}}n^{-\frac{2\beta}{1+2p + 2\beta + 2\eps \theta}}
\end{equation*}
This is somewhat better than   \eqref{e:3dvarminimax}, as there is no logarithmic term, and the factor of $2\eps$ has been replaced by $2\eps \theta$, which can be reduced by taking $\theta$ smaller.  The prefactor will grow, but it is independent of $n$.

    \item In the first case, where $\bar\tau_{\rm b}, \bar\tau_{\rm v}\leq 1$, by taking $\theta$ sufficiently close to zero, we can get arbitrarily close to the optimal rate in \eqref{e:minimax1}.
    
    \item The first case can be realized by taking $t$ and $\beta$ sufficiently small.  The second case, where $\bar\tau_{\rm b}, \bar\tau_{\rm v}> 1$, is accessible by taking $t$ large enough.
    
    \item There are two other cases to consider, $\bar\tau_{\rm b}\leq 1$, $\bar\tau_{\rm v}> 1$ and vice versa, but, for brevity we do not explore them here.
\end{itemize}

In contrast to iterate averaged 3DVAR, there is no gain to iterate averaging for Kalman:
\begin{thm}
\label{t:kalmaniter}
	For the scalar Kalman filter, take $C_0 = \frac{\gamma^2}{\alpha}\sigma>0$.  Then the bias and variance of the iterate-averaged mean, $\bar m_n$ satisfy the inequalities
		\begin{align*}
		|\E[\bar{m}_n] -\ut| &\geq |\E[{m}_n] -\ut|,\\
		\Var(\bar{m}_n) &\geq \Var({m}_n).
		\end{align*}
\end{thm}
Consequently, we do not further explore the impact of averaging upon the Kalman filter in this setting.

\subsection{Outline}
The structure of this paper is as follows. In Section \ref{s:prelim} we
review certain background results needed for our main results.  Section \ref{s:scalar} examines the scalar case, and it includes proofs of Theorems \ref{t:3dvarnoconverge} and \ref{t:kalmaniter}.  We prove Theorems \ref{t:msegeneral1}  and \ref{t:msediagonal1} in Section \ref{s:infinite}.   Numerical examples are given in Section \ref{s:numerics}.  We conclude with a brief discussion in Section \ref{s:disc}.

\medskip

\noindent
{\bf Acknowledgements:} The authors thank A.M. Stuart for suggesting an investigation of this problem.  This work was supported by US National Science Foundation Grant DMS-1818716.  The content of this work originally appeared in \cite{jones2020high} as a part of F.G. Jones's PhD dissertation.  Work reported here was run on hardware supported by Drexel's University Research Computing Facility.

\section{Preliminary Results}

\label{s:prelim}

In this section, we establish some identities and estimates that will be crucial to proving our main results.

Much of our analysis relies on spectral calculus involving the following rational functions which are closely related to the Tikhonov-Phillips regularization $(\alpha+\lambda)^{-1}$:
\begin{align}
\label{e:rdef}
    r_{n,\alpha}(\lambda) &= \paren{\frac{\alpha}{\alpha + \lambda}}^n,\\
\label{e:qdef}    
    q_{n,\alpha}(\lambda) & = \frac{1}{\lambda}\set{1 -  \paren{\frac{\alpha}{\alpha + \lambda}}^n} = \lambda^{-1}(1-r_{n,\alpha}(\lambda)).
\end{align}
These are related by the identity
\begin{equation}
    \label{e:rqrelation}
    \sum_{k=1}^m r_{k,\alpha}(\lambda) = \alpha q_{m,\alpha}(\lambda).
\end{equation}
The following estimates can be found in \cite{Iglesias_Lin_Lu_Stuart_2017} and \cite{Pereverzev}, particularly Section 2.2 of the latter reference:
\begin{lem}
\label{l:rbounds}
For $\lambda \in [0, \Lambda]$ and $n \in \mathbb{N}$,
\begin{gather*}
\label{e:rbound1}
    0<  r_{n,\alpha}(\lambda) \leq \frac{\alpha}{\alpha + n \lambda}\leq 1,\\
\label{e:rbound2}
\lambda^p r_{n,\alpha}(\lambda) \leq \begin{cases}
    \paren{\frac{\alpha p}{n}}^p, & p\in [0,n],\\
    \alpha^n\Lambda^{p-n}, & p >n.
    \end{cases}
\end{gather*}
\end{lem}

\begin{lem}
\label{l:qbounds}
For $\lambda \in [0, \Lambda]$, $n \in \mathbb{N}$,
\begin{gather*}
    \label{e:qbound1}
    \lambda^{p}q_{n,\alpha}(\lambda) \leq 
	\begin{cases}
    \paren{\frac{n}{\alpha}}^{1-p}, & p\in [0,1],\\
    \Lambda^{p-1}, & p > 1,
    \end{cases}\\
    \label{e:qbound2}
    \lambda^{p}q_{n,\alpha}(\lambda) \leq \lambda^{p-1}.
\end{gather*}
\end{lem}

Next, we recall the following result on Hilbert scales,

\begin{prop}
\label{p:hilbertscale}

There exists a constant $D>1$, such that 
for $|\theta|\leq 1$,
\begin{equation*}
    D^{-1}\|x\|_{\theta(1+\nu)}\leq \|(B^\ast B)^{\frac{\theta}{2}}x\|\leq D \|x\|_{\theta(1+\nu)}
\end{equation*}
and 
\begin{equation*}
    \Ran\left((B^\ast B)^{\frac{\theta}{2}}\right )=\Dom\left (\Sigma^{-\frac{\theta(1+\nu)}{2}}_0\right). 
\end{equation*}
\end{prop}
This result, based on a duality argument, is proven in Lemma 4.1 of \cite{Iglesias_Lin_Lu_Stuart_2017}.  See, also, Section 8.4 of \cite{Heinz}, particularly Corollary 8.22.

We also have a few useful identities for the filters which we state without proof.
\begin{lem}
\label{l:KalmanIdent}
For the Kalman filter, the mean and covariance operators and the Kalman gains satisfy the identities

\begin{align*}
    m_n&=\paren{{\gamma^2}{n^{-1}}C_0^{-1}+A^*A}^{-1}\paren{A^*\bar{y}_n + {\gamma^2}n^{-1}C_0^{-1}m_0}\\
        C_n^{-1} &= C_{n-1}^{-1} + \gamma^{-2}  A^\ast A = C_0^{-1} +\gamma^{-2} n A^\ast A\\
        K_n &= (\gamma^2C_{n-1}^{-1} + A^*A)^{-1}A^*=(\gamma^2C_0^{-1} + nA^*A)^{-1}A^* =\gamma^{-2}  C_nA^*.
    \end{align*}
\end{lem}

 	\begin{lem}
    For 3DVAR,
    \begin{equation*}
    \bar{u}_n=\sum_{k=0}^{n-1}\frac{n-k}{n}(I-\mathcal{K}A)^k\mathcal{K}\bar{y}_{n-k} + \sum_{k=0}^{n-1}\frac{1}{n}(I-\mathcal{K}A)^{k}(I-\mathcal{K}A)u_0.
    \end{equation*}
    \end{lem}

\begin{cor}
    \label{c:vbariter}
Letting $v_n = u_n - \ut$, $\bar v_n = \frac{1}{n}\sum_{k=1}^n v_k$,
\begin{equation*}
    \bar v_n = \sum_{k=0}^{n-1}\frac{n-k}{n}(I-\calK A)^{k}\calK \bar\eta_{n-k} + \sum_{k=0}^{n-1}\frac{1}{n}(I - \calK A)^k (I - \calK A)v_0.
\end{equation*}
\end{cor}

\begin{rem}
    \label{r:vbar}
    As this is a linear problem, it will be sufficient to study the behavior of $\bar v_n$ to infer convergence of $\bar u_n$ to $\ut$.
\end{rem}

For the analysis of 3DVAR, the essential decomposition into bias and variance terms can be read off of  Corollary \ref{c:vbariter}.  These can be expressed in the more useful forms using $q_{n,\alpha}$:
\begin{lem}
\label{l:avgident2}
\begin{align}
    \label{e:bias2}
    \bar I_n^{\rm bias} & =   \sum_{k=0}^{n-1}\frac{1}{n}(I - \calK A)^k (I - \calK A)v_0 = \frac{\alpha}{n}\Sigma^{\frac{1}{2}} q_{n,\alpha}(B^\ast B) \Sigma^{\frac{1}{2}} v_0,\\
    \label{e:var2}
    \bar I_n^{\rm var} & = \sum_{k=0}^{n-1}\frac{n-k}{n}(I-\calK A)^{k}\calK \bar\eta_{n-k} = \frac{1}{n}\sum_{j=1}^n { \Sigma^{\frac{1}{2}}q_{n-j+1,\alpha}(B^\ast B) B^\ast}\eta_j.
\end{align}

\end{lem}

\begin{proof}

First, observe that
\begin{equation*}
    I - \mathcal{K}A = \Sigma^{\frac{1}{2}}\alpha(\alpha I + B^\ast B)\Sigma^{-\frac{1}{2}}.
\end{equation*}
Using this in \eqref{e:bias2} together with spectral calculus applied to positive self-adjoint compact operator $B^\ast B$, along with \eqref{e:rqrelation},
\begin{equation*}
\begin{split}
    \bar{I}_n^{\bias} &= \frac{1}{n}\sum_{k=0}^{n-1}\Sigma^{{1}/{2}} \alpha^k (\alpha I + B^\ast B)^{-{k+1}}\Sigma^{-{1}/{2}}_0v_0\\
    &=\frac{1}{n} \sum_{k=1}^n \Sigma^{{1}/{2}} r_{k,\alpha}(B^\ast B) \Sigma^{-{1}/{2}}v_0 = \frac{\alpha}{n} \Sigma^{\frac{1}{2}} q_{n,\alpha}(B^\ast B) \Sigma^{-\frac{1}{2}} v_0.
\end{split}
\end{equation*}
Applying the same computations to \eqref{e:var2}, we have,
\begin{equation*}
\begin{split}
    \bar I_n^{\var} &= \sum_{k=0}^{n-1} \frac{n-k}{n}\alpha^{-1}\Sigma^{\frac{1}{2}}_0 r_{k+1,\alpha}(B^\ast B) B^\ast \bar \eta_{n-k}\\
    & = \frac{1}{n}\sum_{j=1}^n \set{\sum_{k=0}^{n-j}\alpha^{-1} \Sigma^{\frac{1}{2}}r_{k+1,\alpha}(B^\ast B) B^\ast}\eta_j = \frac{1}{n}\sum_{j=1}^n { \Sigma^{\frac{1}{2}}q_{n-j+1,\alpha}(B^\ast B) B^\ast}\eta_j.
\end{split}
\end{equation*}
\end{proof}

\section{Analysis of the Scalar Problem}

\label{s:scalar}

Before studying the general, infinite-dimensional case, it is instructive
to consider the scalar problem, where $X = Y = \R$ and $A$, $\Sigma$, and $\mathcal{K}$ are now scalars.  This setting will also allow us to establish the limitations of both 3DVAR and the Kalman filter.

\subsection{3DVAR}

First, we prove Theorem \ref{t:3dvarnoconverge} which asserts that the 3DVAR iteration cannot converge in mean square:
\begin{proof}
Since $y_n\sim \mathcal{N}(Au^\dagger, \gamma^2)$, we write $y_n = Au^\dagger + \eta_n$ for $\eta_n \sim \mathcal{N}(0,\gamma^2)$. By \eqref{e:3DVAR},
\begin{equation*}
    \begin{split}
        u_n-u^\dagger &= \calK \eta_n + \calK A\ut + (1-\calK A)u_{n-1} - \ut\\
&= \calK \eta_n + (1-\calK A)(u_{n-1} - \ut).
    \end{split}
\end{equation*}
Consequently,
\begin{equation*}
\begin{split}
\E[|u_n-u^\dagger|^2] &= \E[|\calK \eta_n|^2] + \E[|(1-\calK A)(u_{n-1} - \ut)|^2]\\
&\geq \E[|\calK \eta_n|^2] = \calK^2\gamma^2.
\end{split}
\end{equation*}
\end{proof}

Next, studying the bias and variance of the time averaged problem, given by \eqref{e:bias2} and \eqref{e:var2}, we prove
\begin{thm}
	\label{scalarbound}
For scalar time averaged 3DVAR, for $\tau_{\rm b}, \tau_{\rm v}\in [0,1]$
\begin{equation*}
    \E[|\bar u_n - \ut|^2]\leq { (A^2\Sigma)^{-2 \tau_{\rm b}}} |v_0|^2 \paren{\frac{n}{\alpha}}^{-2 \tau_{\rm b}} +\frac{\Sigma \gamma^2}{\alpha}(A^2\Sigma )^{-\tau_{\rm v}}\paren{\frac{n}{\alpha}}^{-\tau_{\rm v}}.
\end{equation*}
\end{thm}
Thus, we have unconditional convergence for any choice for $\alpha>0$, something that we do not have for 3DVAR without any iterate averaging.  The rate of convergence is greatest when $\tau_{\rm b}\geq 1/2$ and $\tau_{\rm v}=1$. 

To obtain the result, we make use of the bias variance decomposition and expressions \eqref{e:bias2} and \eqref{e:var2}.  In the scalar case, $B^\ast B = B^2 = \Sigma A^2$, so that 
\begin{equation}
\label{e:bias1d2}
    \abs{\bar I_n^{\bias}}^2 = \paren{\frac{n}{\alpha}}^{-2}q_{n,\alpha}(\Sigma A^2)^2 |v_0|^2.
\end{equation}
Applying Lemma \ref{l:qbounds} to this expression, we immediately obtain
\begin{prop}
\label{p:bias1dbound}
For $0\leq \tau_{\rm b} \leq 1$,
\begin{equation}
\abs{\bar I_n^{\bias}}^2\leq  { (A^2\Sigma )^{-2 \tau_{\rm b}}} |v_0|^2\paren{\frac{n}{\alpha}}^{-2 \tau_{\rm b}}.
\end{equation}
\end{prop}

For the variance, we have the result
\begin{prop}
\label{p:var1dbound}
Let $\tau_{\rm v}\in [0,1]$,
\begin{equation}
    \E[|\bar{I}_n^{\var}|^2]\leq \frac{\Sigma \gamma^2}{\alpha}(A^2\Sigma )^{-\tau_{\rm v} }\paren{\frac{n}{\alpha}}^{-\tau_{\rm v}}.
\end{equation}
\end{prop}
\begin{proof}

For the scalar case of \eqref{e:var2},using Lemma \ref{l:qbounds},
\begin{equation*}
\begin{split}
    \E[|\bar{I}_n^{\var}|^2]  &= \frac{\gamma^2(A\Sigma)^2}{n^2}\sum_{j=1}^n q_{j,\alpha}(A^2 \Sigma)^2\\
    &=\frac{ \gamma^2 \Sigma}{n^2}(A^2\Sigma)^{1-({1+\tau_{\rm v}})}\sum_{j=1}^n \bracket{(A^2 \Sigma)^{\frac{1+\tau_{\rm v}}{2}} q_{j,\alpha}(A^2 \Sigma)}^2 \\
    &\leq \frac{\Sigma \gamma^2 }{n^2}(A^2\Sigma)^{-\tau_{\rm v}}\sum_{j=1}^n\paren{\frac{j}{\alpha}}^{2\left (1-{\frac{1+\tau_{\rm v}}{2}}\right)}\\
    &\leq \frac{\Sigma \gamma^2 (A^2\Sigma)^{-\tau_{\rm v}}}{n^2} n \paren{\frac{n}{\alpha}}^{1-\tau_{\rm v}} = \frac{\Sigma \gamma^2}{\alpha} (A^2\Sigma)^{-\tau_{\rm v}}\paren{\frac{n}{\alpha}}^{-\tau_{\rm v}}.
\end{split}
\end{equation*}

\end{proof}

\begin{proof}[Proof of Theorem \ref{scalarbound}]
The result then follows immediately by combining the two preceding propositions.

\end{proof}

\subsection{Kalman Filter}

Here, we prove Theorem \ref{t:kalmaniter}, showing there is no improvement in mean squared convergence of Kalman under iterate averaging.

\begin{proof}
Using Lemma \ref{l:KalmanIdent}, for the $k$-the estimate of the mean,
\begin{equation*}
\begin{split}
    m_k &= \paren{\frac{\alpha}{\Sigma k} + a^2}^{-1}\paren{A\bar y_k + \frac{\alpha}{\Sigma k}m_0}\\
    & = \paren{\frac{\alpha}{\Sigma k} + A^2}^{-1}\paren{A^2\ut + A\bar \eta_k + \frac{\alpha}{\Sigma k}m_0}\\
    & =\paren{1 + \frac{\alpha}{A^2 \Sigma k}}^{-1}\ut + \paren{1 + \frac{A^2 \Sigma k}{\alpha}}^{-1}m_0 + \paren{A + \frac{\alpha}{A \Sigma k}}^{-1} \bar \eta_k.
\end{split}
\end{equation*}
and without averaging,
\begin{align*}
    \E[m_n] - \ut & =  \paren{1 + \frac{A^2 \Sigma n}{\alpha}}^{-1}(m_0 - \ut),\\
    \Var(m_n) & = \paren{{A + \frac{\alpha}{A \Sigma n}}}^{-2} \frac{\gamma^2}{n}.
\end{align*}
Then, with averaging, for the bias,
\begin{equation*}
\E[\bar m_n] -\ut = \frac{1}{n}\sum_{k=1}^n \paren{1 + \frac{A^2\Sigma k}{\alpha}}^{-1}(m_0- \ut),
\end{equation*}
and
\begin{equation*}
\begin{split}
    |\E[\bar m_n] -\ut|^2 &= \abs{\frac{1}{n}\sum_{k=1}^n \paren{1 + \frac{A^2\Sigma k}{\alpha}}^{-1}}^2|m_0- \ut|^2\\
    &\geq \abs{\frac{1}{n}\sum_{k=1}^n \paren{1 + \frac{A^2\Sigma n}{\alpha}}^{-1}}^2|m_0- \ut|^2 = |\E[m_n] - \ut|^2.
\end{split}
\end{equation*}
For the variance, first note
\begin{equation*}
\begin{split}
    \bar{m}_n - \E[\bar{m}_n] = \frac{1}{n}\sum_{k=1}^n \paren{A + \frac{\alpha}{A\Sigma k}}^{-1}\bar \eta_k &= \frac{1}{n}\sum_{k=1}^n \paren{A + \frac{\alpha}{A\Sigma k}}^{-1}\set{\sum_{j=1}^k \eta_j}\\
    &= \frac{1}{n}\sum_{j=1}^n \eta_j\set{\sum_{k=j}^n \paren{A + \frac{\alpha}{A\Sigma k}}^{-1}}.
\end{split}
\end{equation*}
Then, by dropping all but the $k=n$-th term in the inner sum,
\begin{equation*}
\begin{split}
    \Var(\bar{m}_n) = \frac{1}{n^2}\sum_{j=1}^n\gamma^2 \set{\sum_{k=j}^n \paren{A + \frac{\alpha}{A\Sigma k}}^{-1}}^2&\geq \frac{1}{n^2}\sum_{j=1}^n\gamma^2 \paren{A + \frac{\alpha}{A\Sigma n}}^{-2}\\
    &\quad = \Var(m_n)
    \end{split}
\end{equation*}

\end{proof}

\section{Analysis of the Infinite Dimensional Problem}

\label{s:infinite}

We return to the bias and variance of 3DVAR in the general, potentially infinite dimensional, setting and obtain estimates on the terms.  We prove the general case in Section \ref{s:general}, and then the diagonal case in \ref{s:diagonal}.  Our minimax results are proven in Section \ref{s:minimax}.

\subsection{General Case}
\label{s:general}

Here, we prove Theorem \ref{t:msegeneral1} by first establishing results on the bias and variance.

\begin{prop}
\label{p:generalbias}
Under Assumption \ref{a1}, with $t \in [0,\nu]$,
\begin{equation*}
\label{e:generalbias}
\|\bar{I}_n^{\bias}\|_t^2\lesssim \paren{\frac{n}{\alpha}}^{-\frac{s+t}{1+\nu}}\|z_0\|^2
\end{equation*}
where $z_0$ solves \eqref{e:auxbias}.
\end{prop}
The fastest possible decay available for the squared bias in  Proposition \ref{p:generalbias} is $\bigo(n^{-2})$ when $s = \nu+2$ and $t=\nu$. 
\begin{proof}
We make use of bias term from Lemma \ref{l:avgident2}, allowing us to write
\begin{equation*}
    \|\bar{I}_n^{\bias}\|_t^2 = \left\|\frac{\alpha}{n}\Sigma^{\frac{t+1}{2}}q_{n,\alpha}(B^\ast B) \Sigma^{-\frac{1}{2}} v_0\right\|^2.
\end{equation*}

Next, we make use of \eqref{e:Hscale} and argue as in the Appendix of \cite{Iglesias_Lin_Lu_Stuart_2017}, applying Proposition \ref{p:hilbertscale}.  Since, by assumption, $v_0 \in \Dom(\Sigma^{-\frac{s}{2}})$, $\Sigma^{-\frac{1}{2}} v_0 \in \Dom(\Sigma^{-\frac{s-1}{2}})$.  Then taking $\theta = (s-1)/(1+\nu)$ in the proposition, $\Sigma^{-\frac{1}{2}} v_0 \in \Ran((B^\ast B)^{\frac{s-1}{2(1+\nu)}})$ allows us to conclude the existence of $z_0$. Therefore, 
\begin{equation*}
    \|\bar{I}_n^{\bias}\|_t^2 = 
    \left\|\frac{\alpha}{n}\Sigma^{\frac{t+1}{2}}q_{n,\alpha}(B^\ast B) (B^\ast B)^{\frac{s-1}{2(1+\nu)}} z_0 \right\|^2.
\end{equation*}
Next, using Proposition \ref{p:hilbertscale} again, now with $\theta = (1+t)/(1+\nu)$,
\begin{equation*}
\begin{split}
    \|\bar{I}_n^{\bias}\|_t^2 &\lesssim \norm{\frac{\alpha}{n}(B^\ast B)^{\frac{t+1}{2(1+\nu)}}q_{n,\alpha}(B^\ast B) (B^\ast B)^{\frac{s-1}{2(1+\nu)}} z_0 }^2\\
    &\quad = \norm{\frac{\alpha}{n}(B^\ast B)^{\frac{s+t}{2(1+\nu)}}q_{n,\alpha}(B^\ast B)z_0}^2\\
    &\leq \paren{\sup_{0\leq\lambda\leq \|B^\ast B\|}\abs{\frac{\alpha}{n} \lambda^{\frac{s+t}{2(1+\nu)}}q_{n,\alpha}(\lambda) }}^{2}\|z_0\|^2\leq \paren{\frac{n}{\alpha}}^{-\frac{s+t}{1+\nu}} \|z_0\|^2.
    \end{split}
\end{equation*}
The last inequality holds since, $s\leq \nu+2$ and $t\leq \nu$, so that $0\leq s+t \leq s+\nu\leq 2\nu + 2$ allowing for the application of Lemma \ref{l:qbounds}.
\end{proof}

\begin{prop}
\label{p:generalvariance}
Under Assumption \ref{a1}, for $t\geq 0$, $\tau_{\rm v} \in [0, 1]$, and for this choice of $\tau_{\rm v}$ and $t$, assume $\Sigma^{(1+t)-\tau_{\rm v}(1+\nu)}$ is trace class. Then
\begin{equation*}
    \E[\|\bar I_n^{\var}\|_t^2]\lesssim \frac{\gamma^2}{\alpha}\Tr(\Sigma^{t+1 - \tau_{\rm v}(1+\nu)}) \paren{\frac{n}{\alpha}}^{-\tau_{\rm v}}.
\end{equation*}
\end{prop}

\begin{rem}
\label{r:generalvar}
The fastest possible decay in the variance will be $\bigo(n^{-1})$ when $\tau_{\rm v} = 1$ and $t$ is sufficiently large such that $\Sigma^{t -\nu}$ is trace class.  However, the bias term requires $t\leq \nu$.  This requires the identity operator to be trace class which will not hold in infinite dimensions.
\end{rem}

\begin{proof}[Proof of Proposition \ref{p:generalvariance}]

We begin with equation \eqref{e:var2} and using that for any bounded operator $T$ and positive self adjoint trace class operator $C$, $|\Tr(CT)|\leq \|T\||\Tr C|$,
\begin{equation*}
\begin{split}
    \E[\|\bar I_n^{\var}\|_t^2] &= \frac{1}{n^2}\sum_{j=1}^n \E[\|{ \Sigma^{\frac{t+1}{2}}q_{n-j+1,\alpha}(B^\ast B) B^\ast}\eta_j\|^2]\\
    &=\frac{\gamma^2}{n^2}\sum_{j=1}^n \Tr\paren{ \Sigma^{\frac{t+1}{2}}q_{j,\alpha}(B^\ast B) (B^\ast B) q_{j,\alpha}(B^\ast B)\Sigma^{\frac{t+1}{2}}}\\
    &=\frac{\gamma^2}{n^2}\sum_{j=1}^n \Tr\paren{ \Sigma^{t+1 - \tau_{\rm v}(1+\nu)} \paren{\Sigma^{\tau_{\rm v}\frac{1+\nu}{2}} (B^\ast B)^{\frac{1}{2}}q_{j,\alpha}(B^\ast B) (B^\ast B)}^2}\\
    &\leq \frac{\gamma^2}{n^2}\sum_{j=1}^n \|{\Sigma^{\tau_{\rm v}\frac{1+\nu}{2}} (B^\ast B)^{\frac{1}{2}}q_{j,\alpha}(B^\ast B) (B^\ast B)}\|^2 \Tr(\Sigma^{t+1 - \tau_{\rm v}(1+\nu)}).
    \end{split}
\end{equation*}
Using Proposition \ref{p:hilbertscale} with $\theta = \tau_{\rm v}$ and Lemma \ref{l:qbounds},
\begin{equation*}
\begin{split}
  \|{\Sigma^{\tau_{\rm v} \frac{1+\nu}{2}} (B^\ast B)^{\frac{1}{2}}q_{j,\alpha}(B^\ast B) (B^\ast B)}\| &\lesssim \| (B^\ast B)^{\frac{1+\tau_{\rm v}}{2}} q_{j,\alpha}(B^\ast B)\|\\
    &\lesssim \sup_{\lambda \in [0, \|B^\ast B\|]} \lambda^{\frac{1+\tau_{\rm v}}{2}} q_{j,\alpha}(\lambda) \lesssim \paren{\frac{j}{\alpha}}^{1 - \frac{1+\tau_{\rm v}}{2}}
    \end{split}
\end{equation*}
Therefore, 
\begin{equation*}
    \begin{split}
        \E[\|\bar I_n^{\var}\|_t^2]  & \lesssim \frac{\gamma^2}{n^2}\Tr(\Sigma^{t+1 - \tau_{\rm v}(1+\nu)})\sum_{j=1}^{n} \paren{\frac{j}{\alpha}}^{1- \tau_{\rm v}}\lesssim \frac{\gamma^2}{\alpha} \Tr(\Sigma^{t+1 - \tau_{\rm v}(1+\nu)}) \paren{\frac{n}{\alpha}}^{-\tau_{\rm v}}
    \end{split}
\end{equation*}

\end{proof}

\begin{proof}[Proof of Theorem \ref{t:msegeneral1}]
The theorem immediately follows from the two preceding propositions.
\end{proof}

\subsection{Simultaneous Diagonalization}
\label{s:diagonal}

A sharper result is available under the simultaneous diagonalization Assumption \ref{a2p} .  
For convenience, letting
\begin{equation}
\label{e:omega1}
    \omega = \frac{1+2\eps}{1+2\eps +2p},
\end{equation}
we have the relationship
\begin{equation}
\label{e:omega2}
    \sigma_i \asymp (\sigma_i a_i^2)^{\omega}.
\end{equation}

\begin{prop}
\label{p:diagbias}
Under Assumption \ref{a2p}, let $\tau_{\rm b}\in [0,1]$ satisfy condition \eqref{e:taucondb}, 
\begin{equation*}
    \norm{\bar{I}_n^{\bias}}_t^2 \lesssim \paren{\frac{n}{\alpha}}^{-2\tau_{\rm b}}.
\end{equation*}
\end{prop}

\begin{proof}
We start with equation \eqref{e:bias2} and then use \eqref{e:omega2} and Lemma \ref{l:qbounds},
\begin{equation*}
\begin{split}
    \norm{\bar{I}_n^{\bias}}_t^2  &= \sum_{i=1}^\infty \inner{\frac{\alpha}{n} \Sigma^{\frac{t+1}{2}} q_{n,\alpha}(B^\ast B) \Sigma^{-\frac{1}{2}} v_0}{\varphi_i}^2\\
    &=\sum_{i=1}^\infty \abs{\frac{\alpha}{n} \sigma_i^{\frac{t}{2}}q_{n,\alpha}(\sigma_i a_i^2)}^2 \abs{v_{0,i}}^2=\paren{\frac{\alpha}{n}}^{2}\sum_{i=1}^\infty \sigma_i^t q_{n,\alpha}(\sigma_i a_i^2)^2\abs{v_{0,i}}^2\\
  & \asymp\paren{\frac{n}{\alpha}}^{-2}\sum_{i=1}^\infty(\sigma_i a_i^2)^{t\omega- 2\tau_{\rm b} } ((\sigma_i a_i^2)^{\tau_{\rm b}}q_{n,\alpha}(\sigma_i a_i^2))^2\abs{v_{0,i}}^2\\
  &\lesssim \paren{\frac{n}{\alpha}}^{-2} \paren{\frac{n}{\alpha}}^{2 - 2\tau_{\rm b}}\sum_{i=1}^\infty(\sigma_i a_i^2)^{t\omega- 2\tau_{\rm b} } \abs{v_{0,i}}^2
\end{split}
\end{equation*}
Using \eqref{e:taucondb}, 
\begin{equation*}
    \begin{split}
        \sum_{i=1}^\infty(\sigma_i a_i^2)^{t\omega- 2\tau_{\rm b} } \abs{v_{0,i}}^2&\asymp\sum_{i=1}^\infty i^{-(1+2\eps + 2p) (t\omega- i\tau_{\rm b}) } \abs{v_{0,i}}^2\\
        &\asymp\sum_{i=1}^\infty i^{-(1+2\eps + 2p) (t\omega- 2\tau_{\rm b}) - 2\beta} i^{2\beta}\abs{v_{0,i}}^2\\
        &\lesssim\paren{\sup_{i} i^{-(1+2\eps + 2p) (t\omega- 2\tau_{\rm b}) - 2\beta}} \sum_{i=1}^\infty i^{2\beta}\abs{v_{0,i}}^2< \infty
    \end{split}
\end{equation*}
we have the result.

\end{proof}
Comparing this to the general case, we again see that if the data is sufficiently smooth and/or we study the probelm in a sufficiently smooth space ($\beta$ and/or $t$ large), we can again obtain $\bigo(n^{-2})$ convergence of the squared bias.

\begin{prop} 
\label{p:diagvar}
Under Assumption \ref{a2p}, having fixed $t$, for $\tau_{\rm v}\in [0,1]$ satisfying \eqref{e:taucondv},
\begin{equation*}
    \E\bracket{\norm{\bar{I}_n^{\var}}_t^2}\lesssim  \frac{\gamma^2 }{\alpha} \paren{\frac{n}{\alpha}}^{-\tau_{\rm v}}
\end{equation*}

\end{prop}

\begin{proof}
Using \eqref{e:var2}, we begin by writing
\begin{equation*}
    \begin{split}
        \E\bracket{\norm{\bar{I}_n^{\var}}_t^2} & = \frac{1}{n^2}\sum_{j=1}^n \E\bracket{\norm{\Sigma^{\frac{t+1}{2}} q_{n-j+1,\alpha}(B^\ast B)B^\ast \eta_j}^2},\\
        & = \frac{\gamma^2}{n^2}\sum_{j=1}^n \Tr\paren{\Sigma^{\frac{t+1}{2}} q_{n-j+1,\alpha}B^\ast B(B^\ast B) q_{n-j+1,\alpha}(B^\ast B)\Sigma^{\frac{t+1}{2}}},\\
        & =  \frac{\gamma^2}{n^2}\sum_{j=1}^n \Tr\paren{\Sigma^{t+1} B^\ast B q_{j,\alpha}(B^\ast B)^2}.
    \end{split}
\end{equation*}
Using \eqref{e:qdef} on each term in the sum,
\begin{equation*}
    \begin{split}
        \Tr\paren{\Sigma^{t+1} B^\ast B q_{j,\alpha}(B^\ast B)^2} & =\sum_{i=1}^\infty \sigma_i^{t+2}a_i^2 q_{j,\alpha}(\sigma_i a_i^2)^2.
    \end{split}
\end{equation*}
Then, using \eqref{e:omega2} and Lemma \ref{l:qbounds}
\begin{equation*}
\begin{split}
    \sigma_i^{t+2}a_i^2 q_{j,\alpha}(\sigma_i a_i^2)^2 &\asymp \sigma_i^{t+1} ((\sigma_i a_i^2)^{\frac{1}{2}} q_{j,\alpha}(\sigma_i a_i^2))^2\\
    &\asymp (\sigma_i a_i^2)^{\omega(t+1)}((\sigma_i a_i^2)^{\frac{1}{2}} q_{j,\alpha}(\sigma_i a_i^2))^2\\
    &\asymp (\sigma_i a_i^2)^{\omega(t+1)-\tau_{\rm v} }((\sigma_i a_i^2)^{(1+\tau_{\rm v})/2} q_{j,\alpha}(\sigma_i a_i^2))^2\\
    &\lesssim (\sigma_i a_i^2)^{\omega(t+1)-\tau_{\rm v} } \paren{\frac{j}{\alpha}}^{1-\tau_{\rm v}}
\end{split}
\end{equation*}
Under assumption \ref{e:taucondv}
\begin{equation}
\label{e:diagvarconst1}
    \sum_{i=1}^\infty (\sigma_i a_i^2)^{\omega(t+1)-\tau_{\rm v} } \asymp\sum_{i=1}^\infty i^{-[(1+2\eps)(t+1) - \tau_{\rm v}(1 + 2\eps + 2p)]}<\infty
\end{equation}
Consequently,
\begin{equation*}
    \Tr\paren{\Sigma^{t+1} (B^\ast B) q_{j,\alpha}(B^\ast B)^2} \lesssim \paren{\frac{j}{\alpha}}^{1-\tau_{\rm v}},
\end{equation*}
and
\begin{equation*}
     \frac{\gamma^2}{n^2}\sum_{j=1}^n \Tr\paren{\Sigma^{t+1} (B^\ast B) q_{j,\alpha}(B^\ast B)^2}\lesssim \frac{\gamma^2}{\alpha} \paren{\frac{n}{\alpha}}^{-\tau_{\rm v}}
\end{equation*}

\end{proof}
In contrast to the non-diagonal case, if the problem is studied in a sufficiently weak sense (large enough $t$), one obtains $\bigo(n^{-1})$ convergence of the variance.
\begin{proof}[Proof of Theorem \ref{t:msediagonal1}]
This result immediately follows from the previous two propositions.

\end{proof}

\subsection{Minimax Analysis}
\label{s:minimax}

\begin{proof}[Proof of Theorem \ref{t:batchminimax1}]
Recall  the spectral cutoff regularization 
\begin{equation*}
    g_\alpha(\lambda)=\lambda^{-1}1_{[\alpha,\infty)}(\lambda).
\end{equation*}
For a fixed $\alpha$, the regularized solution of \eqref{e:batch} is
\begin{equation*}
    \bar{u}_{n,\alpha} = g_{\alpha}(A^\ast A)A^\ast \bar{y}_n = g_{\alpha}(A^\ast A)A^\ast Au^\dagger + g_{\alpha}(A^\ast A)A^\ast \bar{\eta}_n.
\end{equation*}
This allows us to write the bias-variance decomposition of the error as
\begin{equation*}
    \E[\|\bar{u}_{n,\alpha} - u^\dagger\|_t^2] = \|(g_\alpha(A^\ast A)A^\ast A -I)u^\dagger\|_{t}^2 + \E[\|g_\alpha(A^\ast A) A^\ast \bar{\eta}_n\|_t^2]
\end{equation*}
For the bias term, 
\begin{equation*}
\begin{split}
 \|(g_{\alpha}(A^\ast A)A^\ast A -I)u^\dagger\|_t^2  & = \sum_{i=1}^\infty \sigma^t_i (g_\alpha(a_i^2)a_i^2 -1)^2 |u_i^\dagger|^2.
\end{split}
\end{equation*}
Since $a_i^2 \asymp i^{-2p}$, all terms with $i \lesssim \alpha^{-\frac{1}{2p}}$ will vanish.  This leaves us with 
\begin{equation*}
\begin{split}
    \|(g_{\alpha}(A^\ast A)A^\ast A -I)u^\dagger\|_t^2&\lesssim\sum_{i=\lfloor\alpha^{-\frac{1}{2p}}\rfloor}^\infty \sigma_i^t  |u_i^\dagger|^2 \\
    &\lesssim \sum_{i=\lfloor\alpha^{-\frac{1}{2p}}\rfloor}^\infty i^{-t(1+2\eps) - 2\beta} i^{2\beta} |u_i^\dagger|^2\\
    & \lesssim \alpha^{\frac{t(1+2\eps) + 2\beta}{2p}}\sum_i i^{2\beta} |u_i^\dagger|^2.
\end{split}
\end{equation*}
For the variance term, all terms with $i\gtrsim \alpha^{-\frac{1}{2p}}$ will vanish,
\begin{equation*}
    \begin{split}
        \E[\|g_{\alpha}(A^\ast A)A^\ast \eta\|_t^2]& = \frac{\gamma^2}{n}\sum_{i=1}^\infty \sigma_i^{t} g_\alpha(a_i^2)^2 a_i^2\\
        &\lesssim \frac{\gamma^2}{n}\sum_{i=1}^{\lceil\alpha^{-\frac{1}{2p}}\rceil} \sigma_i^{t} a_i^{-2}\lesssim\frac{\gamma^2}{n}\sum_{i=1}^{\lceil\alpha^{-\frac{1}{2p}}\rceil} i^{2p - t(1+2\eps)}\\
        &\lesssim \begin{cases}
         \frac{\gamma^2}{n}\alpha^{-\frac{1 + 2p - t(1+2\eps)}{2p}}& 2p-t(1+2\eps) \neq -1\\
         \frac{\gamma^2}{n} \log\frac{1}{\alpha}& 2p-t(1+2\eps) = -1.
        \end{cases}
    \end{split}
\end{equation*}
Combining the two terms, we thus have,
\begin{equation*}
    \E[\|u- u^\dagger\|_t^2]  \lesssim 
\begin{cases}
         \alpha^{\frac{t(1+2\eps)+2\beta}{2p}}+\frac{\gamma^2}{n}\alpha^{-\frac{1 + 2p - t(1+2\eps)}{2p}}& 2p-t(1+2\eps) \neq -1\\
         \alpha^{\frac{t(1+2\eps)+2\beta}{2p}}+\frac{\gamma^2}{n} \log\frac{1}{\alpha}& 2p-t(1+2\eps) = -1.
        \end{cases}
\end{equation*}
Optimizing over $\alpha$  yields the result.
\end{proof}

\begin{proof}[Proof of Corollary \ref{c:minimaxdiag1}]
Note that in the proof of Proposition \ref{p:diagvar}, under our assumptions, in \eqref{e:diagvarconst1}, with $\tau_{\rm v} = (1-\theta)\bar{\tau}_{\rm v}$
\begin{equation*}
\begin{split}
    \sum_{i=1}^\infty (\sigma_i a_i^2)^{\omega (t+1) - \tau_{\rm v}} &\asymp \sum_{i=1}^\infty i^{-[(1+2\eps)(t+1) - t(1+2\eps)(1-\theta) - 2\eps(1-\theta)]}\\
    &\asymp\sum_{i=1}^\infty i^{-1 - \theta[t(1+2\eps) + 2\eps]}\\
    &\lesssim\int_1^\infty x^{-1 - \theta[t(1+2\eps) + 2\eps]}dx = \frac{1}{\theta[t(1+2\eps) + 2\eps]}
\end{split}
\end{equation*}
Consequently,
\begin{equation*}
\E[\| \bar{u}_n - \ut\|_t^2]\lesssim \paren{\frac{n}{\alpha}}^{-\frac{t(1+2\eps) + 2\beta}{1 + 2\eps + 2p}} + \frac{1}{\theta} \frac{1}{\alpha}  \paren{\frac{n}{\alpha}}^{-(1-\theta)\frac{t(1+2\eps) + 2\eps}{1 + 2\eps + 2p}}
\end{equation*}
where the implicit constants in each term are independent of $\alpha$, $\theta$, and $n$.  The optimally scaled $\alpha$ will be
\begin{equation*}
    \alpha = \theta^{-\frac{1+2p + 2\eps}{1 + 2p + 2\beta + \theta[t(1+2\eps) + 2\eps]}} n^{\frac{2 \beta-2\eps + \theta[t(1+2\eps) + 2\eps]}{1 + 2p + 2\beta + \theta[t(1+2\eps) + 2\eps]}}
\end{equation*}
Substituting back in, we have our result.
\end{proof}

\section{Numerical Experiments}
\label{s:numerics}

In this section we illustrate our results with some numerical experiments.

\subsection{Scalar Examples}

As a simple scalar example, let $A = 1$, $\gamma=0.1$, and $u^\dagger=0.5$.  For 3DVAR, take $u_0 = 0$, $\Sigma = 1$, and $\alpha=1$, while for Kalman, take $m_0 =0$ and $C_0 = 1$.  Running $10^2$ independent trials of each algorithm for $10^4$ iterations, we obtain the results in Figure \ref{fig:scalar}.  These simulations demonstrate our predictions from Theorems \ref{t:3dvarnoconverge}, Theorem \ref{scalarbound}, and Theorem \ref{t:kalmaniter}, that 3DVAR can only converge with time averaging, while Kalman will not be improved by time averaging.  The confidence bounds are computed using $10^4$ bootstrap samples to produce 95\% confidence intervals.

\begin{figure}
    \centering
    \subfigure[]{\includegraphics[width=6.25cm]{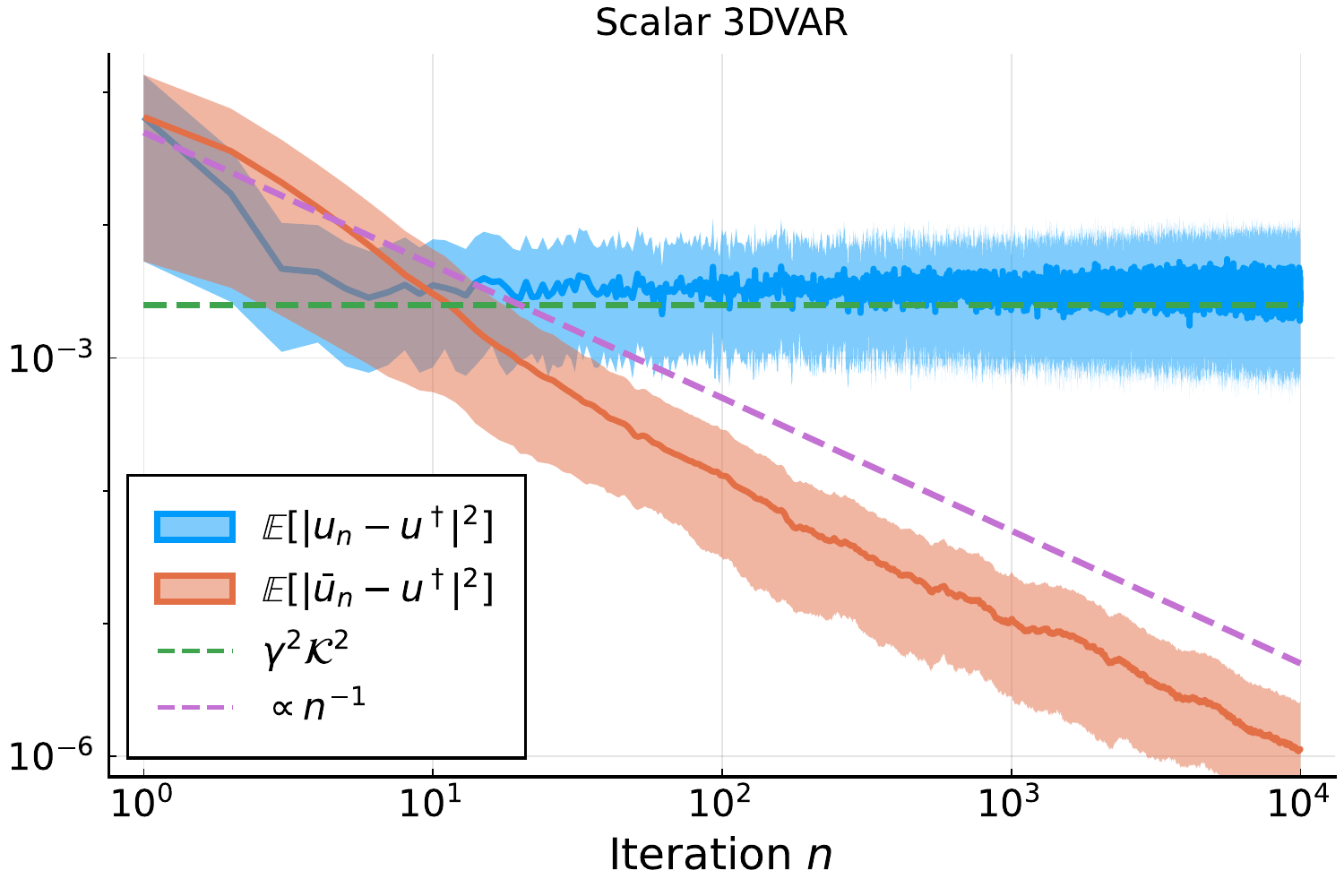}}
    \subfigure[]{\includegraphics[width=6.25cm]{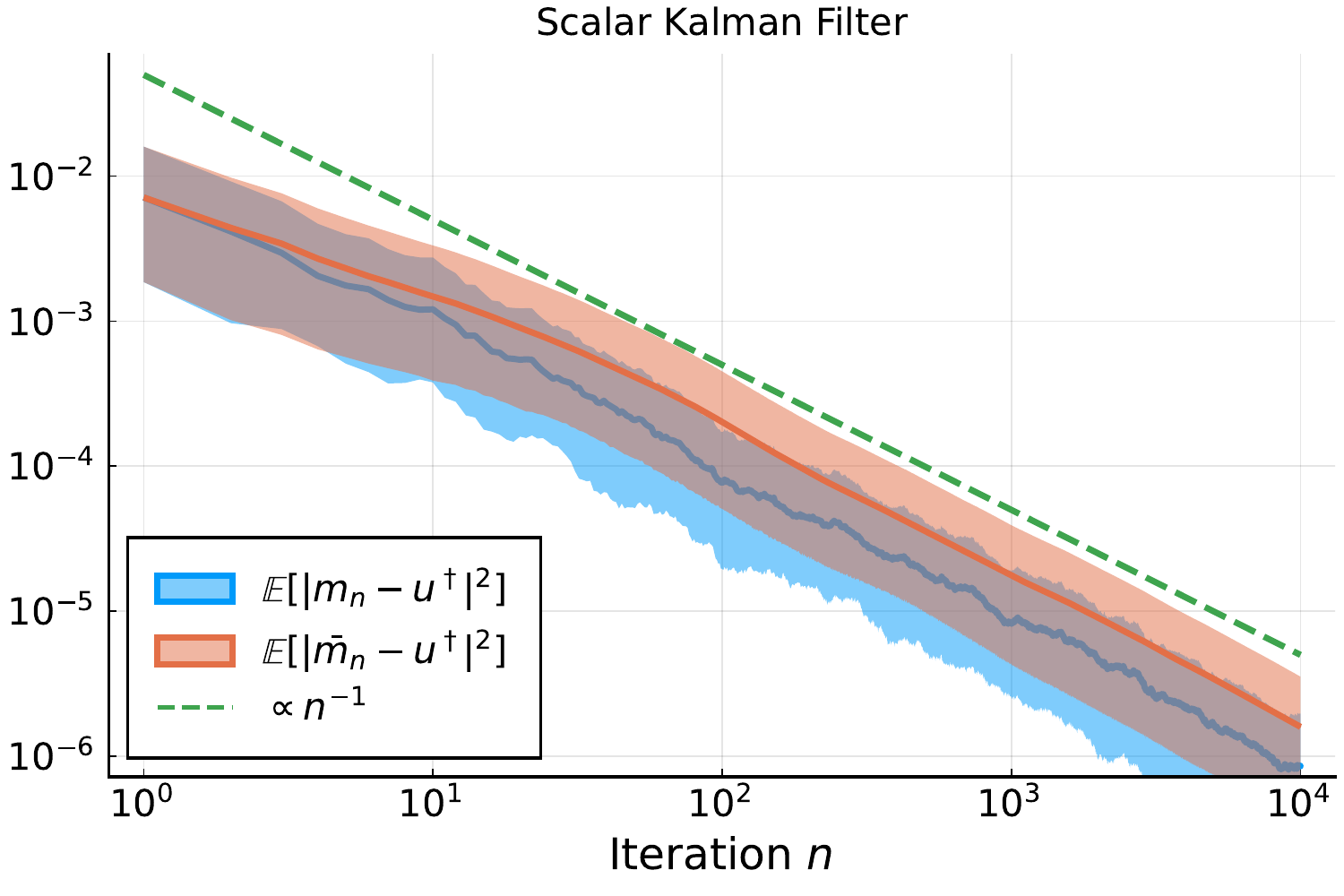}}
    \caption{Scalar results for 3DVAR and the Kalman filter.  These results are consistent with Theorems \ref{t:3dvarnoconverge},\ref{t:kalmaniter}, and \ref{scalarbound}; 3DVAR will not converge without time averaging while Kalman will not improve from time averaging.  Shaded regions reflect 95\% confidence intervals at each $n$.}
    \label{fig:scalar}
\end{figure}

\subsection{Simultaneous Diagonalization Example}
Next, we consdier the case of simultaneous diagonalization, working with functions in $L^2(0,2\pi;\R)$, and
\begin{equation}
\label{e:diagonalprob1}
    A = (I -\tfrac{d^2}{dx^2})^{-1}, \quad \Sigma  = A^2, \quad u^\dagger= 0.
\end{equation}
The $A$ operator is equipped with periodic boundary conditions, allowing us to easily work in Fourier space.  As the problem is linear, we can separately consider the bias and the variance.  In all examples below we discretize on $N = 2^{12}$ modes,  and run for $10^4$ iterations.  This corresponds to $p =2$ and $\epsilon=1.5$ in Assumption \ref{a2p}.

For the bias, we choose, before truncation, as the initial condition 
\begin{equation}
    \label{e:biasic1}
    u_0 = \sum_{k=1}^{\infty} k^{-\frac{1}{2} - \beta - \delta} \cos(k x),
\end{equation}
with $\beta = 1$ and $\delta = 0.01$.  Consequently, this function satisfies \eqref{e:sobolev2} from Assumption \ref{a2p}.  The perturbation $\delta$ is introduced so that we can best see the sharpness of our rates.  Running the truncated and discretized problem, we obtain the results shown in Figure \ref{fig:biasdiag1} for the norms $t=0,0.5, 1, 2$.  As the plots show, we are in good agreement with the maximal rate predicted by Theorem \ref{t:msediagonal1}.

For the variance, taking $u_0 = 0$, we run $10^2$ independent trials of the problem, and then use bootstrapping to estimate 95\% confidence intervals.  The results, shown in Figure \ref{fig:vardiag1}, again show good agreement with the maximal rate predicted by Theorem \ref{t:msediagonal1}.

\begin{figure}
    \centering
    \subfigure[]{\includegraphics[width=6.25cm]{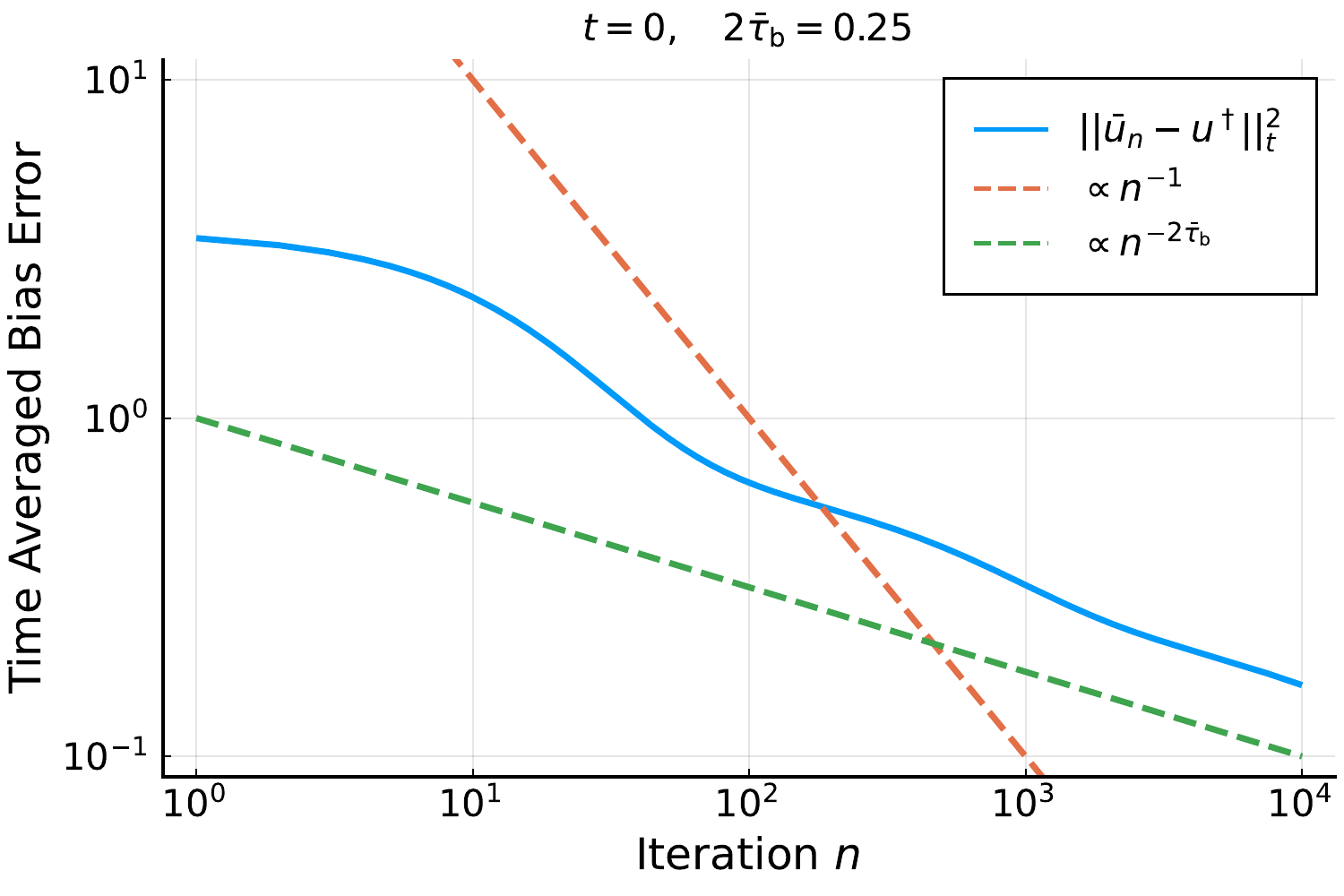}}
        \subfigure[]{\includegraphics[width=6.25cm]{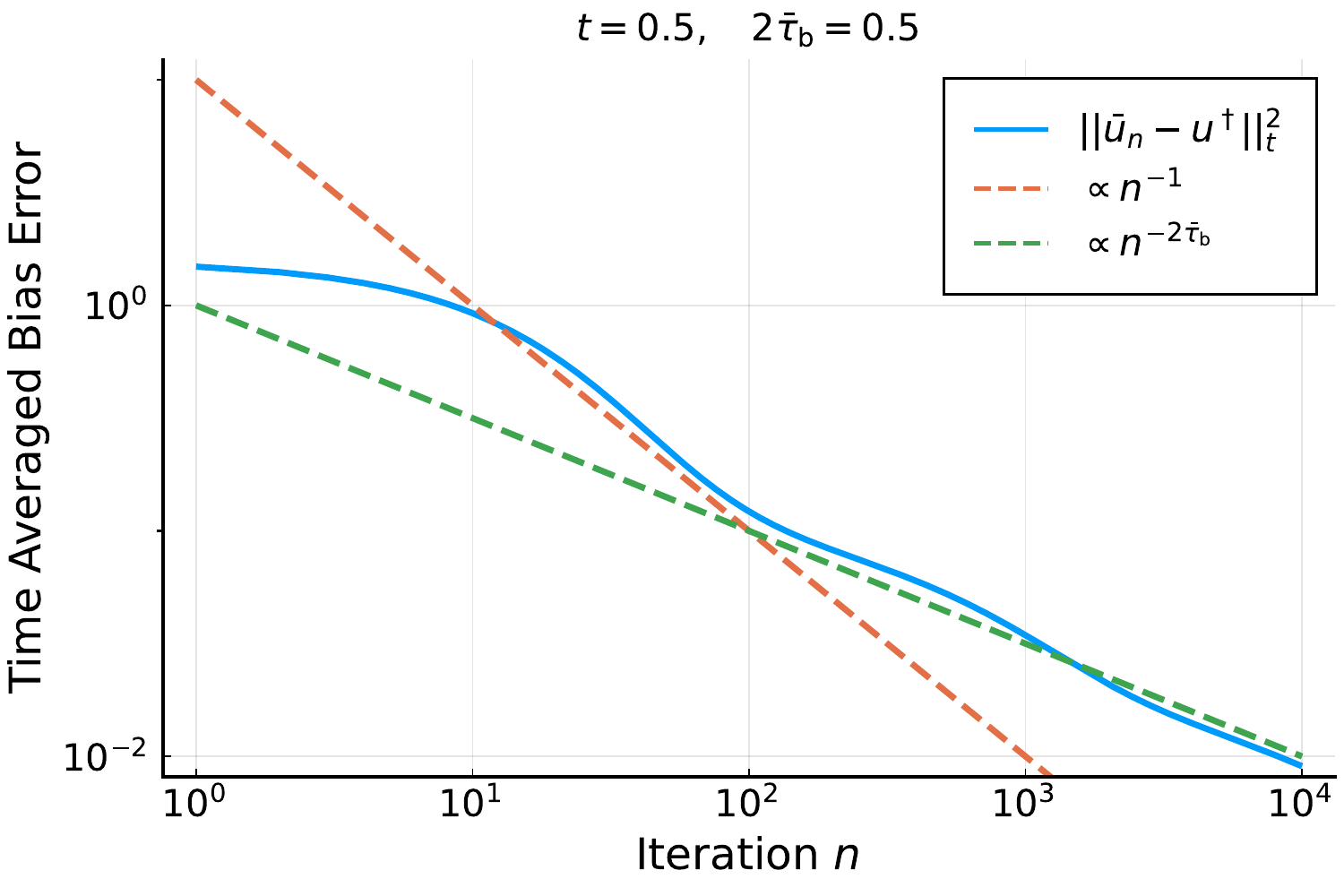}}
        
    \subfigure[]{\includegraphics[width=6.25cm]{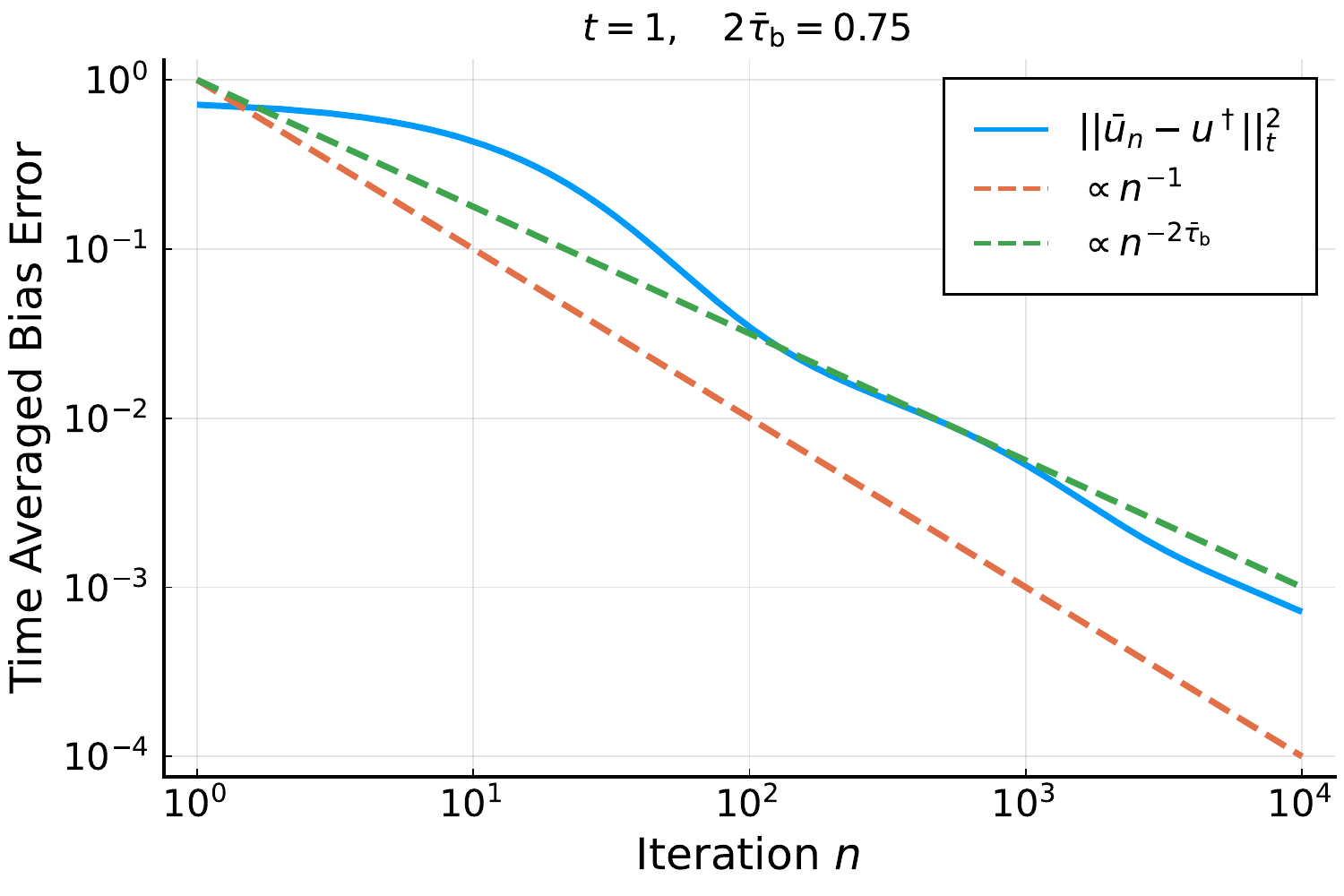}}
        \subfigure[]{\includegraphics[width=6.25cm]{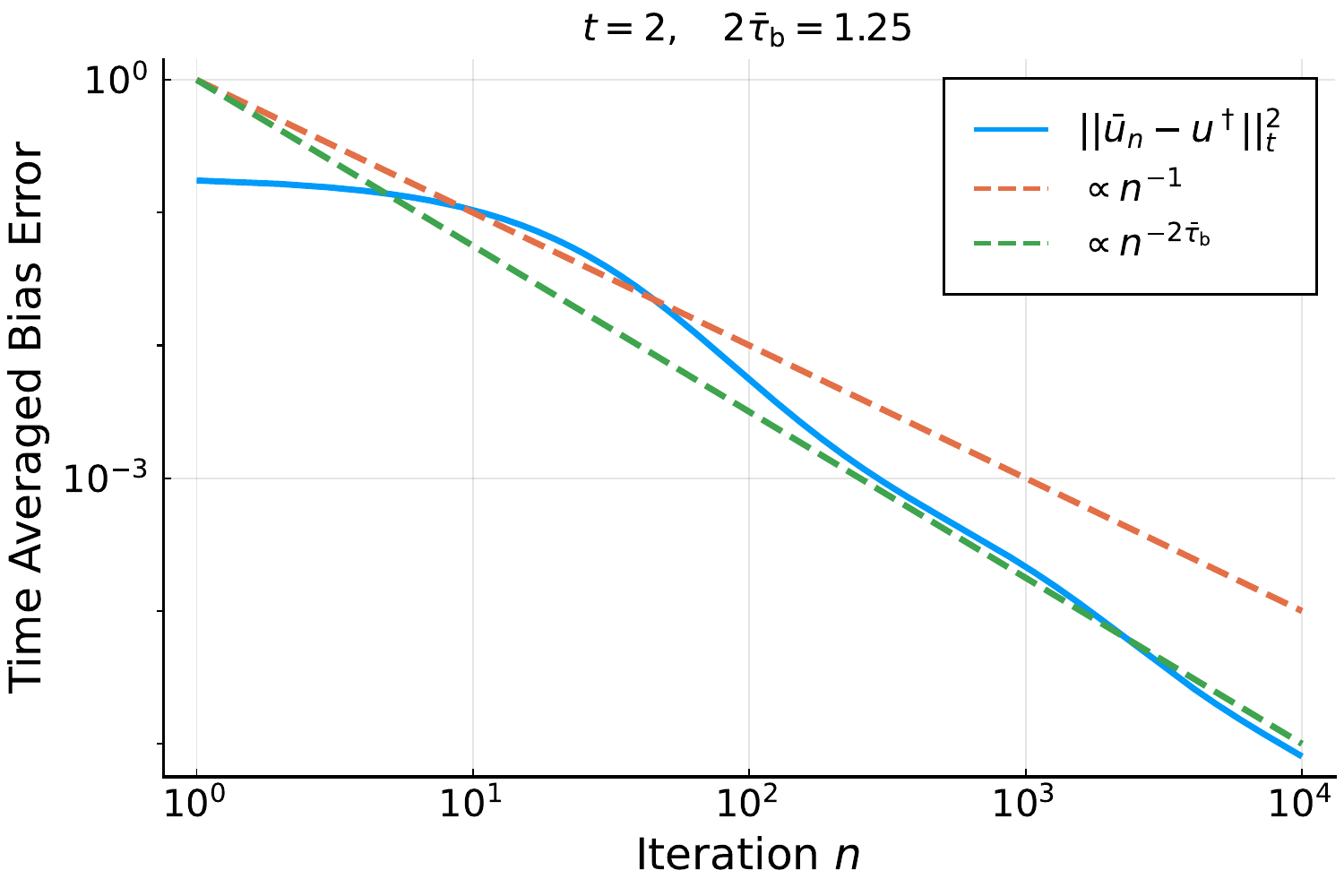}}
    
    \caption{Decay of the squared bias in our simultaneously diagonalized test problem for different $t$-norms.  All are in good agreement with the rates predicted by Theorem  \ref{t:msediagonal1}.  The constant $\bar{\tau}_{\rm v}$ reflects the greatest possible decay rate from \eqref{e:taucondb}.}
    \label{fig:biasdiag1}
\end{figure}

\begin{figure}
    \centering
    \subfigure[]{\includegraphics[width=6.25cm]{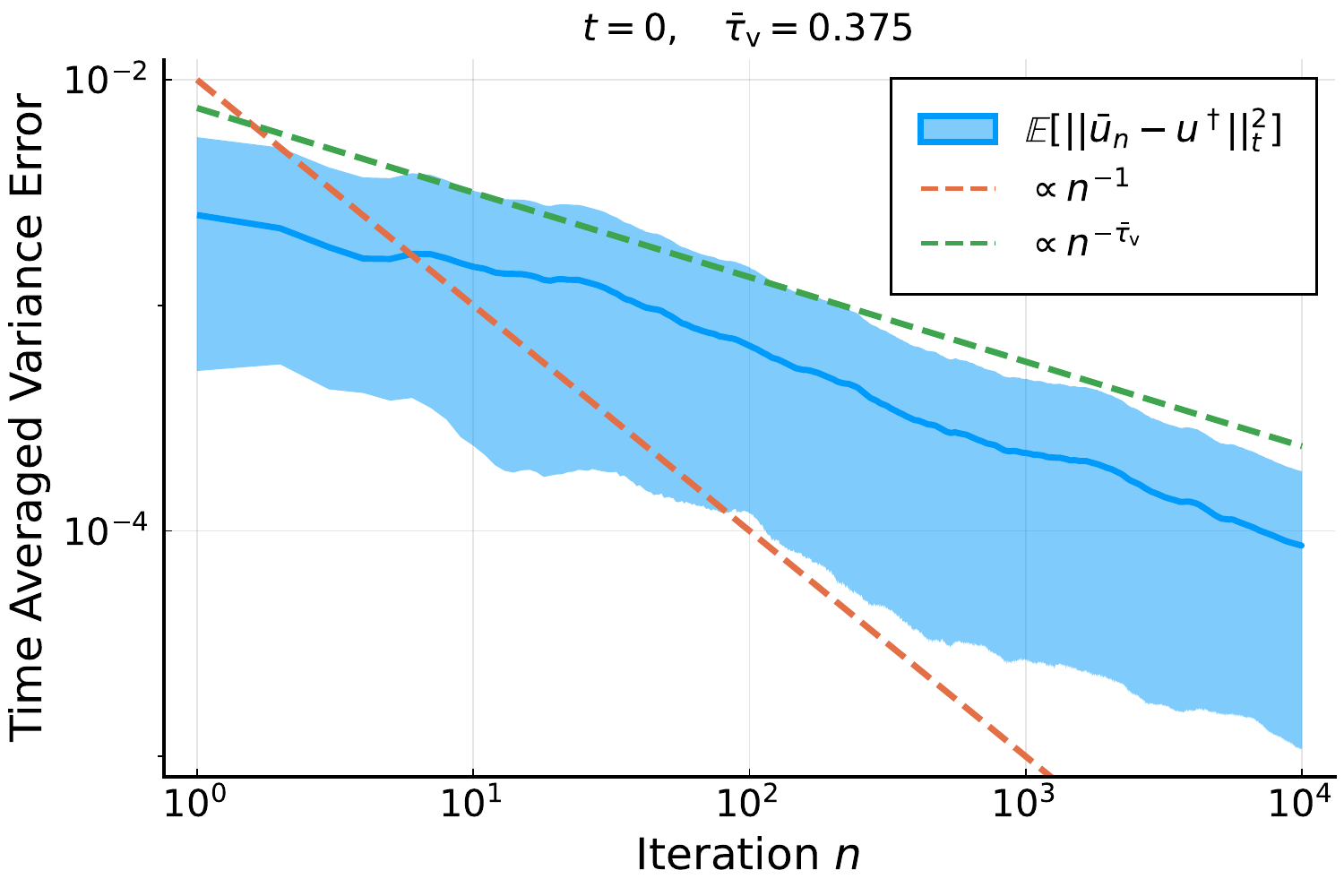}}
    \subfigure[]{\includegraphics[width=6.25cm]{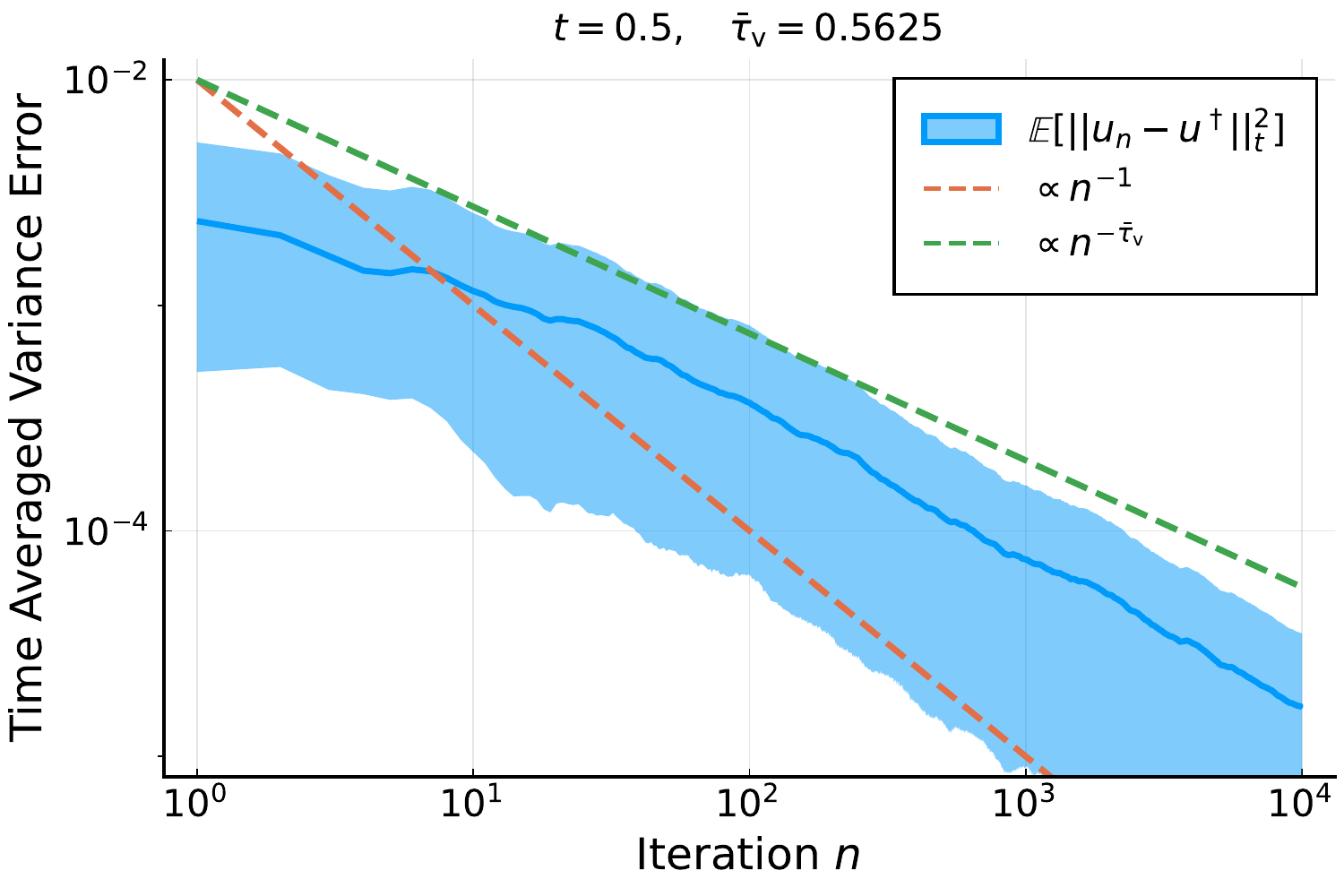}}    
    
    \subfigure[]{\includegraphics[width=6.25cm]{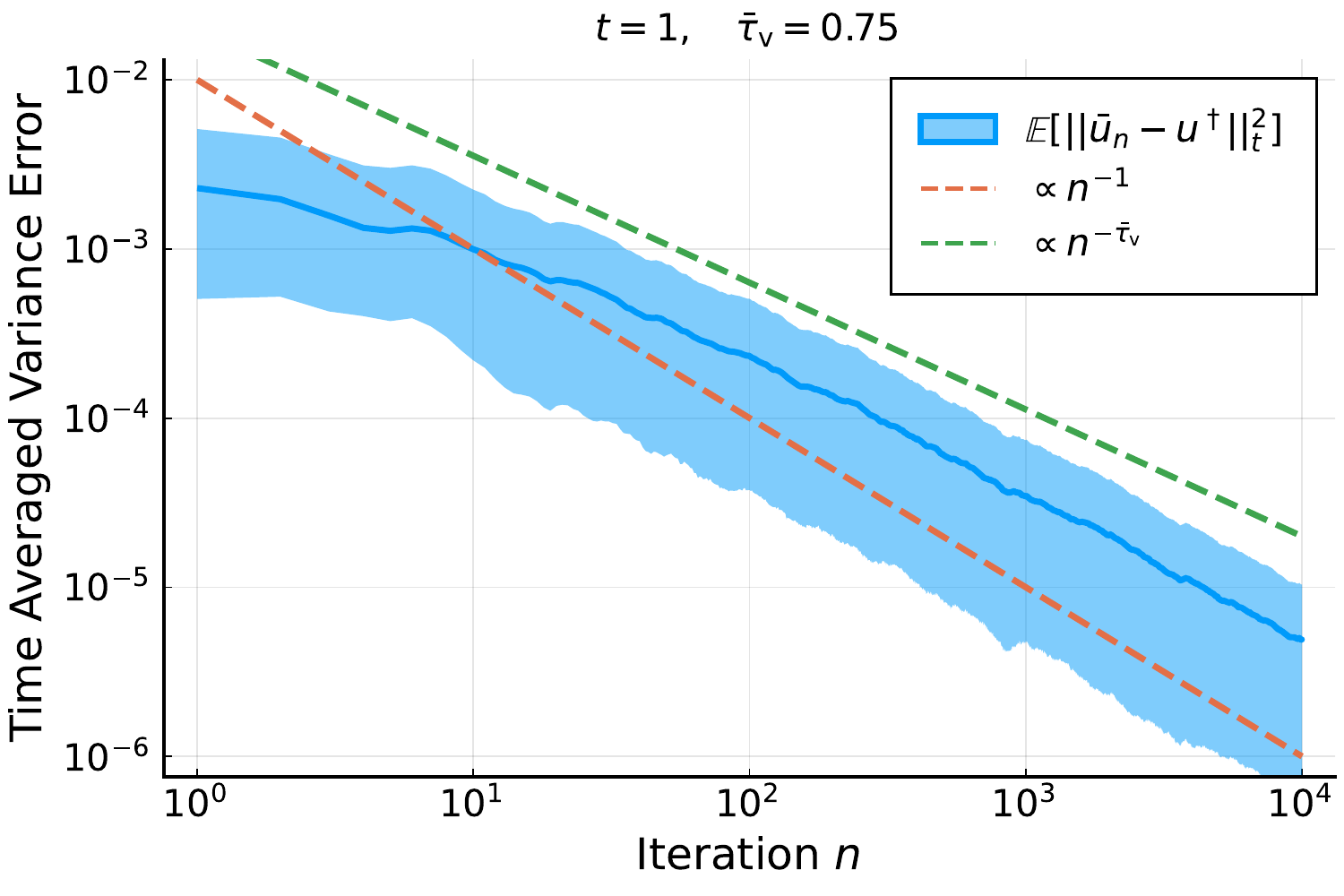}}
    \subfigure[]{\includegraphics[width=6.25cm]{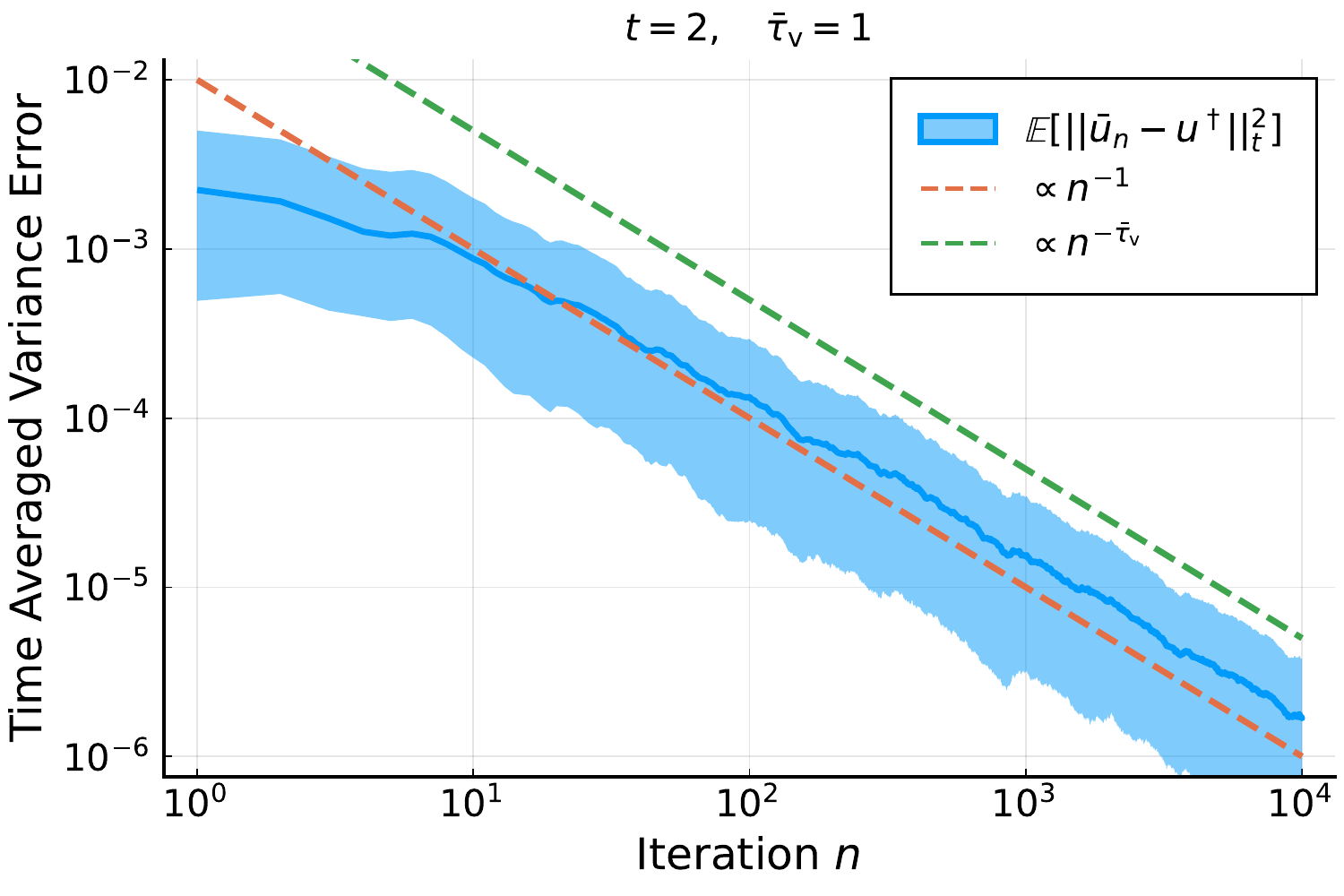}}
    \caption{Decay of the mean squared variance term in our simultaneously diagonalized test problem for different $t$-norms.  All are in good agreement with the rates predicted by Theorem  \ref{t:msediagonal1}.  Shaded regions reflect 95\% confidence intervals at each $n$.  The constant $\bar{\tau}_{\rm v}$ reflects the greatest possible decay rate from \eqref{e:taucondv}.}
    \label{fig:vardiag1}
\end{figure}

\section{Discussion}
\label{s:disc}

In this work we have examined the impact of iterate averaging upon the Kalman filter and 3DVAR as tools for solving a statistical inverse problem.  We have found that this modest post-processing step ensures that the simpler algorithm, 3DVAR, will converge, unconditionally with respect to $\alpha$, in mean square as the number of iterations $n \to\infty$.  In contrast, there is no performance gain when this averaging is applied to the Kalman filter.

Our simulations suggest that our rates, at least in the diagonal case, may be sharp.  For the diagonal case, we should expect to see something slower than the Monte Carlo rate of convergence, $\bigo(n^{-1})$ unless working in a sufficiently weak space (large $t$).  In the general case, it would seem that for the infinite dimensional problem, we will never be able to achieve $\bigo(n^{-1})$ convergence for the reasons outlined in Remark \ref{r:generalvar};  the operator $\Sigma^{t-\nu}$ would need to be trace class, but $t\leq \nu$ for the bias to converge.  The sharpness of the result in the non-diagonalizable case remains to be established.  There is also potential for the extension of this work to the analogous continuous in time problem \eqref{e:continv} studied in \cite{lu2021asymptotical}.

In actual applications, the problem will always be finite dimensional, making $\bigo(n^{-1})$ achievable.  In a spectral Galerkin formulation, truncating to $N$ modes, and, $\Tr\Sigma^{t-\nu}_N$, will always be finite, though the constant may be large.   Hence, we should expect to see  $\bigo(n^{-1})$ convergence, for sufficiently large $n$ and a sufficiently severe dimensional truncation.

\section*{Acknowledgements} The authors thank A.M. Stuart for suggesting an investigation of this problem.  This work was supported by US National Science Foundation Grant DMS-1818716.  The content of this work originally appeared in \cite{jones2020high} as a part of F.G. Jones's PhD dissertation.  Work reported here was run on hardware supported by Drexel's University Research Computing Facility.

\bibliographystyle{abbrv}
\bibliography{refs}

\end{document}